\documentclass{amsart}

\usepackage[textsize=tiny,textwidth=4cm]{todonotes}
\usepackage{lipsum}
\usepackage{amsmath}
\usepackage{amssymb}
\usepackage{amsfonts}
\usepackage{mathtools}
\usepackage{array}
\usepackage{multirow}
\usepackage{capt-of}
\usepackage[skip=10pt,font=footnotesize]{caption}
\usepackage{subcaption}
\usepackage[shortcuts]{extdash}
\usepackage{graphicx}
\usepackage{epstopdf}
\usepackage{algorithmic}
\usepackage{listings}
\usepackage[
  algo2e,
  algosection,
  linesnumbered,
  ruled
]{algorithm2e}
\usepackage[textsize=tiny]{todonotes}
\usepackage{bbm}
\usepackage{dutchcal} 
\usepackage{colortbl}
\usepackage{wrapfig}
\usepackage{placeins}
\usepackage[bottom=3.5cm,left=3.6cm,right=3.8cm]{geometry}
\usepackage{hyperref}
\hypersetup{bookmarks,
	raiselinks,
	colorlinks,
	citecolor=black,
	linkcolor=black,
	urlcolor=black,
	filecolor=black,
	menucolor=black}
\usepackage[capitalize,nameinlink]{cleveref}
\crefname{section}{section}{sections}
\crefname{subsection}{subsection}{subsections}
\Crefname{section}{Section}{Sections}
\Crefname{subsection}{Subsection}{Subsections}
\crefformat{equation}{\textup{#2(#1)#3}}
\crefrangeformat{equation}{\textup{#3(#1)#4--#5(#2)#6}}
\crefmultiformat{equation}{\textup{#2(#1)#3}}{ and \textup{#2(#1)#3}}
{, \textup{#2(#1)#3}}{, and \textup{#2(#1)#3}}
\crefrangemultiformat{equation}{\textup{#3(#1)#4--#5(#2)#6}}%
{ and \textup{#3(#1)#4--#5(#2)#6}}{, \textup{#3(#1)#4--#5(#2)#6}}{, and \textup{#3(#1)#4--#5(#2)#6}}

\Crefformat{equation}{#2Equation~\textup{(#1)}#3}
\Crefrangeformat{equation}{Equations~\textup{#3(#1)#4--#5(#2)#6}}
\Crefmultiformat{equation}{Equations~\textup{#2(#1)#3}}{ and \textup{#2(#1)#3}}
{, \textup{#2(#1)#3}}{, and \textup{#2(#1)#3}}
\Crefrangemultiformat{equation}{Equations~\textup{#3(#1)#4--#5(#2)#6}}%
{ and \textup{#3(#1)#4--#5(#2)#6}}{, \textup{#3(#1)#4--#5(#2)#6}}{, and \textup{#3(#1)#4--#5(#2)#6}}

\makeatletter
\DeclareRobustCommand*{\bbl@ap}[1]{\textormath{\textsuperscript{#1}}{^{\mathrm{#1}}}}%
\DeclareRobustCommand*{\bbl@ped}[1]{\textormath{$_{\mbox{\fontsize\sf@size\z@ \selectfont#1}}$}{_\mathrm{#1}}}%
\let\ap\bbl@ap
\let\ped\bbl@ped
\makeatother

\newtheorem{theorem}{Theorem}[section]

\newtheorem{proposition}[theorem]{Proposition}
\newtheorem{lemma}[theorem]{Lemma}
\newtheorem{assumption}[theorem]{Assumption}
\newtheorem{remark}[theorem]{Remark}
\newtheorem{corollary}[theorem]{Corollary}

\newcommand{\takeout}[1]{} 

\newcommand{\Oe}{\Omega_{e}}
\newcommand{\Oi}{\Omega_{i}}
\newcommand{\pOi}{\partial\Omega_{i}}
\newcommand{\pOe}{\partial\Omega_{e}}
\newcommand{\mO}{\Omega}
\newcommand{\mOe}{\Omega_{e}}
\newcommand{\me}{e}
\newcommand{\mOip}{\Omega_{i}^{\prime}}

\newcommand{\bsA}{\boldsymbol{A}} 
\newcommand{\bsR}{\boldsymbol{R}} 
\newcommand{\bsE}{\boldsymbol{E}}
\newcommand{\bsI}{\boldsymbol{I}}
\newcommand{\bsM}{\boldsymbol{M}} 
\newcommand{\bsP}{\boldsymbol{P}}
\newcommand{\bsS}{\boldsymbol{S}}
\newcommand{\bsT}{\boldsymbol{T}}
\newcommand{\bsu}{\boldsymbol{u}} 
\newcommand{\bsv}{\boldsymbol{v}} 
\newcommand{\bsf}{\boldsymbol{f}}
 
\newcommand{\pOmega}{\Omega^{\prime}} 

\definecolor{petrol}{RGB}{0,110,137}
\definecolor{lila}{RGB}{108,0,228}

\newcommand{\dist}{\operatorname{dist}} 
\newcommand{\spanlin}{\operatorname{span}}

\definecolor{grau}{RGB}{174,170,153}

\DeclareMathOperator*{\argmin}{arg\,min}

\newcommand{\lila}[1]{{\color{black} #1}}
\def\AH{\color{black}}

\usepackage{amsopn}
\DeclareMathOperator{\diag}{diag}

\newcommand{\R}{\mathbb{R}} 

\ifpdf
  \DeclareGraphicsExtensions{.eps,.pdf,.png,.jpg}
\else
  \DeclareGraphicsExtensions{.eps}
\fi

\ifpdf
\hypersetup{
  pdftitle={A fully algebraic and robust two-level Schwarz method based on optimal local approximation spaces},
  pdfauthor={A. Heinlein and K. Smetana}
}
\fi

\title[A fully algebraic Schwarz method]{A fully algebraic and robust two-level Schwarz method based on optimal local approximation spaces}

\author{Alexander Heinlein}
\address{Delft Institute of Applied Mathematics, Delft University of Technology, The Netherlands. {a.heinlein@tudelft.nl}}

\author{Kathrin Smetana}
\address{Department of Mathematical Sciences, Stevens Institute of Technology, 1 Castle Point Terrace, Hoboken, NJ 07030, United States of America. {ksmetana@stevens.edu}
}
\date{\today}

\subjclass[2010]{65F08, 65F10, 65N55, 65N30, 68W10}

\keywords{Domain decomposition methods, multiscale methods, overlapping Schwarz preconditioner, adaptive coarse spaces}

\begin{document}
	
	\begin{abstract}
Two-level domain decomposition {\AH preconditioners} lead to {\AH fast} convergence and scalability of iterative solvers. However, for highly heterogeneous problems, \lila{where the coefficient function is varying rapidly on several possibly non-separated scales}, the condition number {\AH of the preconditioned system} generally depends on the contrast \lila{of the coefficient function} leading to a deterioration of convergence. Enhancing the methods by coarse spaces  constructed from suitable local \lila{eigenvalue} problems, also denoted as adaptive or spectral coarse spaces, restores robust, contrast-independent convergence. However, these eigenvalue problems typically rely on non-algebraic information, such that the adaptive coarse spaces cannot be constructed from the fully assembled \lila{system matrix}. 

\lila{{\AH In this paper, a novel algebraic adaptive coarse space, which relies} on the $a$-orthogonal decomposition of (local) finite element (FE) spaces into functions that solve the partial differential equation (PDE) with {\AH some} trace and FE functions that are zero on the boundary, is proposed. {\AH In particular, the basis is constructed from eigenmodes of two types of local}
eigenvalue problems {\AH associated with the edges of the domain decomposition}. To approximate functions that solve the PDE locally, we employ a transfer eigenvalue problem, which has originally been proposed for the construction of optimal local approximation spaces for multiscale methods. In addition, we make use of a Dirichlet eigenvalue problem that is a slight modification of the Neumann eigenvalue problem used in the adaptive generalized Dryja--Smith--Widlund (AGDSW) coarse space. Both eigenvalue problems rely solely on local Dirichlet matrices, which can be extracted from the fully assembled system matrix. By combining arguments from multiscale and domain decomposition methods we derive a contrast-independent upper bound for the condition number. }

The robustness of the method is confirmed {\AH numerically} for a variety of {\AH heterogeneous} coefficient distributions, including binary random distributions and a coefficient function \lila{constructed from the SPE10 benchmark}. The results are comparable to those of the non-algebraic AGDSW coarse space, also for those cases where {\AH the convergence of the classical} \lila{algebraic} generalized Dryja--Smith--Widlund (GDSW) coarse space {\AH deteriorates}. Moreover, \lila{the} coarse space dimension is the same or comparable to the AGDSW coarse space for all numerical experiments.
	\end{abstract}
	
	\maketitle

\section{Introduction}\label{sect:intro}

Domain decomposition methods {\AH (DDMs)} are a popular class of methods that yield rapid convergence in the iterative solution of \lila{linear systems of equations} arising from partial differential equations (PDEs). 
In particular, if a suitable coarse level is used, {\AH DDMs} have proved to be scalable for a wide range of problems, which, for elliptic problems, can be shown theoretically by proving \lila{an upper bound} for the condition number.

Unfortunately, in presence of strong heterogeneities in certain problem parameters, {\AH the} convergence of classical {\AH DDMs} may deteriorate. For instance, for a diffusion problem with a coefficient function \lila{that} is varying rapidly on possibly several non-separated scales, the condition number may depend on the contrast of the maximum and minimum {\AH values of the} \lila{coefficient.} One way to overcome this issue is by using adaptive coarse spaces, also known as spectral coarse spaces. These approaches are based on solving local generalized eigenvalue problems and selecting a number of eigenfunctions based on a user-chosen \lila{tolerance} for the eigenvalues. The selected functions are used to construct coarse basis functions with local support. Due to the use of spectral information, these coarse spaces typically yield a \lila{provable upper bound of the condition number that is independent of the contrast and depends on the tolerance for the eigenvalues.} Hence, adaptive coarse spaces yield robust convergence. A variety of adaptive coarse spaces has been introduced for nonoverlapping {\AH DDMs}~\cite{MS07,SR13,KRR15a,KKR15,ZAM16,BPS17,OWZD17,CW16,KCW17,PD17}, 
most of which consider finite element tearing and interconnect -- dual primal (FETI--DP)
and balancing domain decomposition by constraints (BDDC)
methods, and overlapping Schwarz methods~\cite{Galvis:2010:Domab,Galvis:2010:Domaa,Dolean:2012:Anala,Spillane:2014:Absta,Heinlein:2018:Multb,Heinlein:2018:AGD,Heinlein:2019:AGD,Heinlein:2020:AGD,Eikeland:2018:Overa,Gander:2015:Anala,knepper:2022:dissertation,Bastian:2022:MSD}. 

Even though provably robust, none of the approaches referenced above is algebraic. {\AH This} means that {\AH none of the}
above-mentioned preconditioners 
{\AH can}
be constructed from the fully assembled system matrix \lila{therefore requiring new assembly routines and potentially even access to the mesh of the FE discretization}. FETI-DP and BDDC {\AH methods} are generally not algebraic since they require local Neumann matrices on the subdomains, which cannot be extracted from \lila{the system matrix}. Schwarz methods can be constructed algebraically if the coarse space can be constructed algebraically. However, the adaptive coarse spaces mentioned above all require additional information for the definition of the eigenvalue problems, such as local Neumann matrices or geometric information. 

\lila{In this paper, we propose, to the best of our knowledge for the first time, an {\AH interface}-based adaptive coarse space for two-level overlapping Schwarz methods that is both fully algebraic and robust. Relying on the well-known $a$-orthogonal decomposition of local FE spaces into functions that solve the PDE numerically with 
{\AH a prescribed trace}
as a boundary condition and FE functions that are zero on the boundary, we propose in this paper building the adaptive coarse space from two local eigenvalue problems {\AH associated with} each edge of the domain decomposition. To approximate functions that solve the PDE locally, we employ the transfer eigenvalue problem introduced in~\cite{SmePat16}, which is known from the construction of optimal local approximation spaces \cite{BabLip11, SmePat16, MaScDo22, schleuss2022optimal} for a novel type of multiscale methods that allows full error control even for heterogeneous problems with non-separated scales \cite{BabLip11, MalPet14, MaScDo22, OwZ11,OZB13, Owh17, SmePat16}. For the approximation of the functions with zero trace we make use of a Dirichlet eigenvalue problem which is a slight modification of the Neumann eigenvalue problem used in the non-algebraic adaptive generalized Dryja--Smith--Widlund (AGDSW)~\cite{Heinlein:2018:AGD,Heinlein:2019:AGD,Heinlein:2020:AGD,knepper:2022:dissertation} coarse space. The adaptive coarse space is then build from energy-minimizing extensions of the eigenfunctions. Our new method is algebraic as both eigenvalue problems rely solely on local Dirichlet matrices, which can be extracted from the fully assembled system matrix. 
We show that using and combining arguments from these novel type of multiscale methods and DDMs allows deriving a contrast-independent upper bound for the condition number. 

In~\cite{KorYse16,KoPeYs18}, techniques from Schwarz methods have been used to develop and analyze such novel type of multiscale methods, highlighting a connection between the latter and certain DDMs.

The adaptive coarse space proposed in this paper} belongs to a class of adaptive coarse spaces which first partition the interface into nonoverlapping components and compute the eigenvalue problems on these components; cf.\ the spectral harmonically enriched multiscale (SHEM)~\cite{Gander:2015:Anala,Eikeland:2018:Overa}, overlapping Schwarz -- approximate component synthesis (OS-ACMS)~\cite{Heinlein:2018:Multb}, and AGDSW coarse spaces. All these approaches yield a minimum number of total degrees of freedom in all local generalized eigenvalue problems since no degree of freedom appears in more than one eigenvalue \lila{problem. Our new method} is most closely related to the AGDSW method. Even though the AGDSW coarse space contains the generalized Dryja--Smith--Widlund (GDSW) coarse space~\cite{Dohrmann:2008:DDL,Dohrmann:2008:FEM}, which can be constructed algebraically, it is not algebraic since local Neumann matrices appear in the eigenvalue problems. 

\lila{Algebraic coarse spaces for overlapping Schwarz methods have recently, and in parallel to the {\AH preparation of this manuscript,} 
also been proposed} by Gouarin and Spillane~\cite{Gouarin:2021:FAD,Spillane:2021:TNF} as well as {\AH Al Daas et al.}~\cite{al2019class,Daas:2021:RAM,daas2022efficient}. In both cases, the methods can be seen as extensions of the generalized eigenproblems in the overlaps (GenEO)~\cite{Spillane:2014:Absta} approach, where local eigenvalue problems in the {\AH overlaps or the overlapping subdomains} of the 
Schwarz method are solved to compute the coarse space. 
{\AH In order to obtain algebraic coarse spaces, the authors mostly focus on general linear algebra arguments, such as matrix splittings and the Sherman–Morrison–Woodbury formula; see also~\cite{spillane2021abstract,agullo2019robust} for abstract descriptions of the GenEO framework. In contrast, our approach is based on a priori knowledge about the elliptic PDE.}
We conjecture that the proposed methodology can be readily extended to other elliptic problems such as linear elasticity, parabolic problems, and higher spatial dimensions{\AH; see also the extensions of the related AGDSW method~\cite{Heinlein:2019:AGD,Heinlein:2020:AGD}.}

Note that algebraic multigrid (AMG) methods~\cite{ruge1987algebraic} are another famous class of algebraic solvers for {\AH linear systems of equations}, and spectral information has also been used to improve their robustness, for instance, in the spectral element-based algebraic multigrid ($\rho$AMGe) method~\cite{Chartier:2003:SAM}.

{\AH The paper is organized as follows:} We propose adaptive coarse spaces for DDMs that are fully algebraic in \cref{sect:new method} and derive a bound for the \lila{condition number in \cref{sect:bound cond number}}. Beforehand, we briefly {\AH introduce our heterogeneous diffusion model problem in~\cref{sect:problem} and} review two-level overlapping Schwarz preconditioner in \cref{subsect:sect:two-level Schwarz} and adaptive coarse spaces in \cref{subsect:adaptive coarse spaces}. In particular, we also elaborate in \cref{sect:Motivation} on the challenges that arise when one wishes to construct adaptive coarse spaces without relying on local discrete variational problems with Neumann boundary conditions that would require new assembly routines on the local subdomains. \lila{Finally, we discuss the computational realization of the proposed method in \cref{sec:implementation} and demonstrate its robustness numerically in \cref{sec:results}.}

\section{Problem Setting}\label{sect:problem}

Let $\Omega\subset\mathbb{R}^2$ be a bounded domain with Lipschitz boundary and $\alpha \in L^{\infty}(\Omega)$ with $0 < \alpha_{min} \leq \alpha \leq \alpha_{max} < \infty$ be a highly heterogeneous coefficient function, possibly with high jumps. We consider the variational problem: 
\begin{align} 
\text{Find } &\mathcal{u} \in H_0^1(\Omega)	: \qquad a_{\Omega}(\mathcal{u},\mathcal{v}) = \mathcal{f}(\mathcal{v}) \qquad \forall \mathcal{v} \in H_0^1(\Omega),\label{eq:var}\\
\nonumber \mbox{where\quad}
	a_{\Omega}( \mathcal{u}, \mathcal{v}) &:= \int_\Omega\alpha(x) (\nabla  \mathcal{u}(x))^T \nabla  \mathcal{v}(x)\,dx 
	\quad 
	\mbox{and} 
	\quad	
	 \mathcal{f}( \mathcal{v}) := \int_\Omega  \mathcal{f}(x) \mathcal{v}(x)\,dx,
\end{align}
respectively, and $ \mathcal{f} \in L^2(\Omega)$. We equip $H^{1}_{0}(\Omega)$ with the energy norm $|\mathcal{u}|_{a_{\Omega}}:= ( a_{\Omega}(\mathcal{u},\mathcal{u}) )^{1/2}$. \lila{Due to space limitations, we defer {\AH a discussion of} the treatment of Neumann or mixed boundary conditions to a forthcoming paper.}

Let $\tau_h$ be a quasi-uniform triangulation of $\Omega$ into triangles or quadrilaterals with element size $h$. \lila{To simplify the presentation, we assume that the triangulation resolves the coefficient function, i.e., that $\alpha$ is constant on each element.} Then, we introduce a conforming finite element (FE) space $V_{\Omega}^{0} \subset H_0^1(\Omega)$ of dimension $N_{\Omega}$, \lila{where, for the sake of simplicity, we consider piece-wise linear (P1) or bilinear (Q1) FE spaces. 
We} obtain the following discrete variational problem:
\begin{equation} \label{eq:disc_var}
\text{Find } u \in V_{\Omega}^{0}: \qquad a_{\Omega}(u,v) = f(v) \qquad \forall v \in V_{\Omega}^{0},
\end{equation}
where $f(v):= \int_{\Omega} \mathcal{f}(x) v(x)\, dx$ for $v \in V_{\Omega}^{0}$. The algebraic version of \eqref{eq:disc_var} then reads 
\begin{equation}\label{eq:alg}
\text{Find } \bsu \in \R^{N_{\Omega}}:	 \bsA \bsu =\bsf, \quad \text{where } \bsA \in \R^{N_{\Omega}\times N_{\Omega}}, \bsf \in \R^{N_{\Omega}}.
\end{equation}

\section{Adaptive Coarse Spaces for Two-Level Overlapping Schwarz Preconditioners}\label{sect:two-level Schwarz}

\subsection{Two-Level Overlapping Schwarz Preconditioner}\label{subsect:sect:two-level Schwarz}

Let $\Omega$ be decomposed into nonoverlapping subdomains $\Omega_i$ {\AH with maximal diameter $H$} \lila{such that
$
\bar{\Omega} = \bigcup_{i=1}^{M} \bar{\Omega}_{i},$ $\Omega_{i} \cap \Omega_{j} = \emptyset \text{ for } i \neq j.
$
We} assume that the boundaries of the subdomains are Lipschitz continuous \lila{ and do not intersect any element of $\tau_{h}$}. The domain decomposition interface is given as
\begin{equation} \label{eq:interface}
	\Gamma = \bigcup_{i\neq j} \left( \partial\Omega_i \cap \partial\Omega_j \right) \setminus \partial\Omega.
\end{equation}
Let then $\{\Omega^{\prime}_{i}\}_{i=1}^{M}$ be a corresponding overlapping decomposition of $\Omega$ with overlap $\delta\geq h$. We introduce associated  conforming FE spaces $V_{\pOmega_{i}}^{0} \subset H^{1}_{0}(\pOmega_{i})$, $i=1,\hdots,M$, and introduce operators $R_{\Omega \to \mOip}: V_{\Omega}^{0} \rightarrow V_{\pOmega_{i}}^{0}$ that \emph{restrict} FE functions in $V_{\Omega}^{0}$ to $V_{\pOmega_{i}}^{0}$. The operators $E_{\mOip\to \Omega}: V_{\pOmega_{i}}^{0} \rightarrow V_{\Omega}^{0}$ \emph{extend} FE functions in $V_{\pOmega_{i}}^{0}$ by zero to FE functions in $V_{\Omega}^{0}$ accordingly. Note that the indices of the restriction and extension operators $R_{\Omega \to \mOip}$ and $E_{\mOip\to \Omega}$ here and henceforth show the domain of the respective FE spaces the operators map from and to.\footnote{Very often the operators $R_{\Omega \to \mOip}$ and $E_{\mOip\to \Omega}$ are denoted by $R_{i}$ and $R_{i}^{T}$ in the literature{\AH; see, e.g., \cite{Toselli:2005:DDM})}. As we will be needing also additional restriction and extension operators e.g., for FE spaces on the edges or $\Gamma$, we indicate the domains of the associated FE spaces here.}

Next, we introduce local bilinear forms $a_{\pOmega_{i}}: V_{\pOmega_{i}}^{0} \times V_{\pOmega_{i}}^{0} \rightarrow \R$ and corresponding local stiffness matrices $\bsA_{\pOmega_{i}}$, $i=1,\hdots, M$. We use exact local solvers and thus obtain
\begin{equation}\label{eq:def local bil form}
a_{\pOmega_{i}}(u_{i},v_{i}) = a_{\Omega}(E_{\mOip\to \Omega}u_{i}, E_{\mOip\to \Omega}v_{i}) \quad \forall u_{i}, v_{i} \in V_{\pOmega_{i}}^{0}, \quad  \bsA_{\pOmega_{i}} = \bsR_{\Omega \to \mOip}\bsA \bsE_{\mOip\to \Omega},
\end{equation}
 for $i=1,\hdots,M$, where $\bsR_{\Omega \to \mOip}$ and $\bsE_{\mOip\to \Omega}=\bsR_{\Omega \to \mOip}^{T}$ are the algebraic counterparts of $R_{\Omega \to \mOip}$ and $E_{\mOip\to \Omega}$, respectively{\AH; cf.}~\cite{Toselli:2005:DDM}. 
Following, e.g., \cite{Toselli:2005:DDM,QuaVal99} we introduce operators $\widetilde{P}_{i}: V_{\Omega}^{0} \rightarrow  V_{\pOmega_{i}}^{0}$, defined as
\begin{equation}\label{eq:def Schwarz op}
a_{\pOmega_{i}}(\widetilde{P}_{i}u,v_{i}) = a_{\Omega}(u,E_{\mOip\to \Omega}v_{i}), \quad \text{for } v_{i} \in  V_{\pOmega_{i}}^{0}, \quad i=1,\hdots,M.
\end{equation}
We may then define projections
\begin{equation}
P_{i} = E_{\mOip\to \Omega}\widetilde{P}_{i}: V_{\Omega}^{0} \rightarrow E_{\mOip\to \Omega}V_{\pOmega_{i}}^{0} \subset V_{\Omega}^{0}, \quad i=1,\hdots,M,
\end{equation}
with algebraic counterparts $\bsP_{i}=\bsE_{\mOip\to \Omega} \bsA_{\pOmega_{i}}^{-1} \bsR_{\Omega \to \mOip} \bsA$, $i=1,\hdots,M${\AH; see}~\cite{Toselli:2005:DDM}{\AH. The} additive Schwarz operator $P_{AS-1}:= \sum_{i=1}^{M} P_{i}$ {\AH then} reads in matrix form as
\begin{equation}\label{eq:1-level-precond}
\bsP_{AS-1} := \sum_{i=1}^{M} \bsE_{\mOip\to \Omega}\bsA_{\pOmega_{i}}^{-1}\bsR_{\Omega \to \mOip} \bsA = \sum_{i=1}^{M} \bsR_{\Omega \to \mOip}^{T}\bsA_{\pOmega_{i}}^{-1}\bsR_{\Omega \to \mOip} \bsA.
\end{equation}
This Schwarz operator is a preconditioned operator $\bsM_{AS-1}^{-1} \bsA$ consisting of \lila{the one-level Schwarz preconditioner 
$
	\bsM_{AS-1}^{-1} = \sum_{i=1}^{M} \bsE_{\mOip\to \Omega}\bsA_{\pOmega_{i}}^{-1}\bsR_{\Omega \to \mOip}
$
and the system matrix $\bsA$. 
{\AH In}
this one-level Schwarz method, information is only exchanged between neighboring subdomains {\AH through the overlaps, and as a result, the convergence generally deteriorates} 
for a large number of subdomains{\AH; see}~\cite{QuaVal99}. {\AH As a remedy,} one may add a global coarse space $X_{0} \subset V_{\Omega}^{0}$.}\footnote{In many publications, the coarse space is denoted by $V_{0}$. However, to avoid confusion with FE spaces including functions with zero trace, we denote the coarse space here by $X_{0}$.} Correspondingly, we introduce an
{\AH interpolation} operator $E_{0}: X_{0} \rightarrow V_{\Omega}^{0}$, which expresses functions in $X_{0}$ in the FE-basis of $V_{\Omega}^{0}$. The columns of the algebraic counterpart $\bsE_{0}$ therefore contain the FE coefficients of the basis functions of $X_{0}$. Introducing the corresponding restriction operator $R_{0}: V_{\Omega}^{0} \rightarrow X_{0}$ and its algebraic counterpart $\bsR_{0}=\bsE_{0}^T$, we can define $P_{0} = E_{0}\widetilde{P}_{0}: V_{\Omega}^{0} \rightarrow E_{0}X_{0} \subset V_{\Omega}^{0}$. We define the operator $\widetilde{P}_{0}$ analogously to \cref{eq:def Schwarz op}, where
\begin{equation}\label{eq:def:bil form 0}
a_{0}(u_{0},v_{0}) = a_{\Omega}(E_{0}u_{0}, E_{0}v_{0}) \quad \forall u_{0}, v_{0} \in X_{0}, \quad
 \bsA_{0} = \bsR_{0}\bsA \bsE_{0}= \bsR_{0}\bsA \bsR_{0}^{T};
\end{equation}
this means that we also consider the case of an exact coarse solver.
We then define the two-level Schwarz operator~\cite{Toselli:2005:DDM} \lila{
$
P_{AS-2} := \sum_{i=0}^{M} P_{i},
$
and its algebraic counterpart takes the form}
\begin{equation} \label{eq:two-level_schwarz}
	\bsP_{AS-2} = \bsM_{AS-2}^{-1} \bsA := \underset{\text{coarse level}}{\underbrace{\bsE_{0}\bsA_{0}^{-1}\bsR_{0} \bsA}} + \underset{\text{first level}}{\underbrace{\sum_{i=1}^{M} \bsE_{\mOip\to \Omega}\bsA_{\pOmega_{i}}^{-1}\bsR_{\Omega \to \mOip} \bsA}}.
\end{equation}

Different choices of coarse spaces $X_0$ yield numerically scalable preconditioners, that is, a condition number bound which is independent of the number of subdomains. As a result, the convergence of iterative solvers, such as Krylov {\AH subspace} methods, with such a two-level Schwarz preconditioner is independent of the number of subdomains as well. However, using a standard Lagrangian {\AH FE} basis{\AH,} for $X_0$, we obtain the condition number bound
\begin{equation} \label{eq:cond_contrast}
	\kappa\left( \bsP_{AS-2} \right)
	=
	\kappa\left( \bsM_{AS-2}^{-1} \bsA \right) 
	\leq 
	C \max\limits_{T \in \tau_H} \max\limits_{x,y \in \omega_T} \left(\frac{\alpha\left(x\right)}{\alpha\left(y\right)}\right) \left( 1+\frac{H}{\delta}\right)
\end{equation}
for our model problem~\cref{eq:var}{\AH; similar bounds hold for other classical (non-adaptive) coarse spaces. Here, $\tau_H$ is the coarse triangulation, which, in our case, coincides with the nonoverlapping domain decomposition $\left\lbrace \Omega_{i} \right\rbrace_{i=1,\ldots,N}$. Moreover,} $\omega_T$ corresponds to the union of all coarse mesh elements which touch a coarse mesh element $T$. Sharper variants of this estimate can be derived, but the dependence on the contrast of the coefficient function remains{\AH; see~\cite{Graham2007}}. This means that the convergence of a Krylov \lila{subspace method preconditioned} with this two-level Schwarz preconditioner might actually depend on the contrast of the coefficient function $\alpha$, resulting in very high iteration counts; cf. also the results in~\cref{sec:results}.

\subsection{Adaptive Coarse Spaces}\label{subsect:adaptive coarse spaces} 

\Cref{fig:lagrangian_interpolation} attempts to visualize {\AH one of} the reasons that the coefficient contrast {\AH may arise} in the condition number bound. The energy 
\begin{equation} \label{eq:energy}
	\left| \cdot \right|_{a,\Omega}^2 := a_{\Omega} (\cdot,\cdot)
\end{equation}
of the function \lila{$u$ plotted in \cref{fig:lagrangian_interpolation} (left) depends only on the minimum value $\alpha_{\text{min}}$ of the coefficient function $\alpha$ but is independent of the maximum value $\alpha_{\text{max}}$ as its gradient is zero (green) in the yellow region, where $\alpha=\alpha_{\text{max}}$.} If we interpolate the function with piecewise bilinear Lagrangian basis functions \lila{(see \cref{fig:lagrangian_interpolation} (right)) the interpolant $u_{0}$ decays to zero within the yellow region (red), and hence, its energy clearly depends on $\alpha_{\text{max}}$. Therefore,} in any energy estimate of the coarse interpolation, which is part of {\AH the proof of a condition number bound} in the abstract Schwarz theory~\cite{Toselli:2005:DDM},
\begin{equation} \label{eq:est_coarse_interp}
	\left| u_0 \right|_{a,\Omega}^2 \leq C \left| u \right|_{a,\Omega}^2,
\end{equation}
the constant $C$ has to depend on the contrast of $\alpha$, $\alpha_{\text{max}}/\alpha_{\text{min}}$. Therefore, in order to obtain a robust condition number bound, the coefficient function has to be taken into account in the \lila{construction of the} coarse space{\AH; as we will also observe in~\cref{sec:results}, it is additionally necessary to add one coarse basis function for each high coefficient component crossing the domain decomposition interface.}

\begin{figure}
	\centering
	\includegraphics[width=0.32\textwidth]{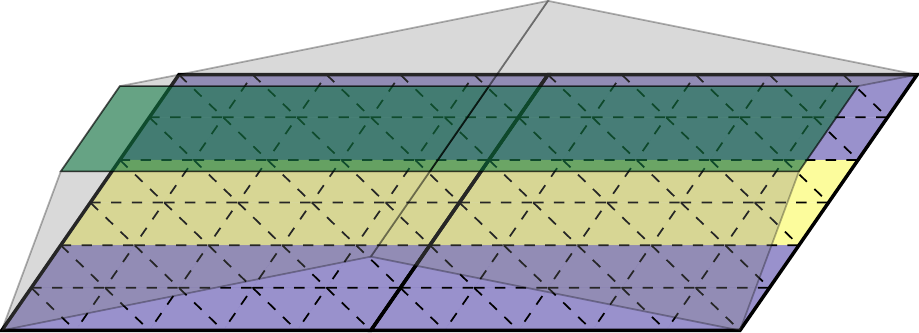}
	\includegraphics[width=0.32\textwidth]{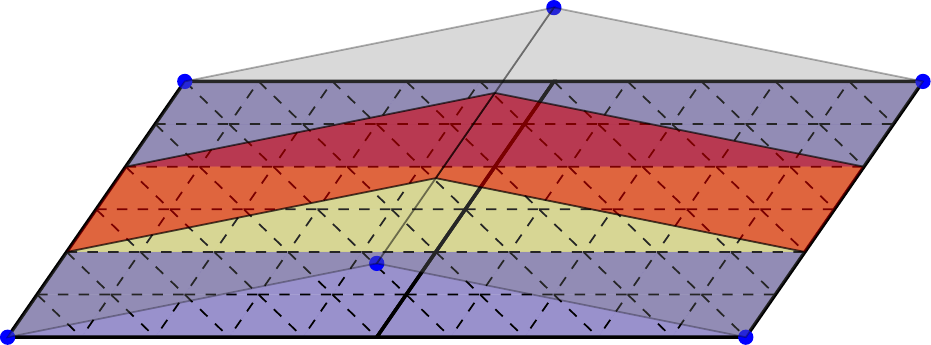}
	\caption{For a heterogeneous coefficient function, a Lagrangian coarse interpolation (here, piecewise {\AH bi}linear; right) of a low-energy function (left) may have a high energy: Let the yellow part of the two (nonoverlapping) subdomains correspond to a high coefficient ($\alpha_{\text{max}}$) and the blue part to a low coefficient ($\alpha_{\text{min}}$). Since the function itself is constant (marked in green) in the high coefficient region but varying in the remaining part, the energy depends on $\alpha_{\text{min}}$ but not on $\alpha_{\text{max}}$. The piecewise linear interpolation has a nonzero gradient everywhere, such that the energy depends on $\alpha_{\text{max}}$ (marked in red). In this case, the stability constant of the interpolation depends on the contrast of the coefficient function.
		\label{fig:lagrangian_interpolation}
	}
\end{figure}

Adaptive coarse spaces account for variations in the coefficient functions by including \lila{eigenfunctions} of local eigenvalue problems into the coarse space\lila{, which is why they are also denoted as \textit{spectral coarse spaces}.} The term \textit{adaptive} stems from the fact that all eigenvalues below a certain tolerance $tol$ can be included in the coarse space, resulting in a condition number bound of {\AH the form}
\begin{equation} \label{eq:cond_bound}
	\kappa\left( \bsP_{AS-2} \right)
	=
	\kappa\left( \bsM_{AS-2}^{-1} \bsA \right)
	\leq
	C \left( 1 + \frac{1}{tol} \right),
\end{equation}
where $C$ is independent of the contrast of the coefficient function. Hence, the number {\AH of} coarse basis functions does not have to be determined beforehand, but it is chosen adaptively 
based on $tol$. 

In this paper, we develop a new energy-minimizing adaptive coarse space \lila{consisting of basis functions} which minimize the energy {\AH $\left| \cdot \right|_{a,\Oi}^2$} on each nonoverlapping subdomain{\AH. Energy-minimizing functions are part of constructing a coarse space with a contrast-independent energy~\cref{eq:est_coarse_interp}. Moreover, constructing the coarse basis by energy-minimization is one of the key ingredients for algebraicity since it does not require the availability of a coarse triangulation. In particular,} the energy-minimizing extension $v_i = H_{\pOi \rightarrow \Oi} (v_{\pOi})$ from $\pOi$ to $\Oi$ for a FE function $v_{\pOi}$ on $\pOi$ {\AH is defined as follows:}
given some boundary values ${\AH v_{\pOi}}$ on the interface, the corresponding 
{\AH extension} $H_{\pOi \rightarrow \Oi} ({\AH v_{\pOi}})$ solves
\begin{equation} \label{eq:energy_min_ext}
v_i
=
\argmin_{v|_{\pOi} = {\AH v_{\pOi}}} 
\left| v \right|_{a,\Oi}^2
\ 
\Leftrightarrow
\ 
\begin{array}{rcl}
	a_{\Oi} (v_i,w_i) & = & 0 \quad \forall w_i \in V_{\Oi}^0, \\
	v_i & = & {\AH v_{\pOi}} \quad \text{on } \pOi.
\end{array}
\end{equation}
An energy-minimizing 
{\AH extension}
$v = H_{\Gamma \rightarrow \Omega}(v_{\Gamma})$ extends interface values $v_\Gamma$ into the interior of each subdomain, with zero Dirichlet boundary conditions on $\partial\Omega$. 
In matrix form, this corresponds to
\begin{equation} \label{eq:harm_ext}   
\bsv
=
\begin{pmatrix}	
	-\bsA_{II}^{-1} \bsA_{I \Gamma} \\
	\bsI_{\Gamma}
\end{pmatrix}
\bsv_\Gamma,
\end{equation}
where we make use of the splitting of the rows and columns corresponding to interior ($I$) and interface ($\Gamma$) nodes
$$
\bsA
=
\begin{pmatrix}
\bsA_{II} & \bsA_{I\Gamma} \\
\bsA_{\Gamma I} & \bsA_{\Gamma\Gamma}
\end{pmatrix}
$$
{\AH and $\bsI_{\Gamma}$ corresponds to the identity matrix on the degrees of freedom of $\Gamma$.}
Here, as usual, the Dirichlet boundary {\AH degrees of freedom} are counted as interior degrees of freedom. Due to the Dirichlet boundary conditions in $\bsA$, {\AH the boundary conditions are then enforced automatically.}
Note that $\bsA_{II} = \diag\left( \bsA_{\Oi,II} \right)$ is a block-diagonal matrix, where $\bsA_{\Oi,II}$ is the matrix corresponding to the interior degrees of freedom in $\Oi$. Hence, in a parallel setting, $\bsA_{II}^{-1}$ can be applied independently and concurrently for the subdomains. Moreover, we have that
$
	\left| v \right|_{a,\Oi}^2
	=
	a_{\Oi} (v,v)
	=
	\bsv^T \bsA \bsv
	=
	\bsv_{\Gamma}^T \bsS \bsv_{\Gamma},
$
with the Schur complement $\bsS = \bsA_{\Gamma \Gamma} - \bsA_{\Gamma I} \bsA_{II}^{-1} \bsA_{I \Gamma}$.

Let us now discuss the general idea of adaptive coarse spaces which are based on an interface partition. 
To that end, let the interface $\Gamma$ be partitioned into edges and vertices. 
\lila{We then} solve a generalized eigenvalue problem of the form
\begin{equation} \label{eq:evp1}
\text{Find } \bsv \in \mathbb{R}^{N_{\mathring{e}}} \enspace \text{such that }
	\bsS_{e} \, \bsv
	= 
	\mu  	
	\bsA_{\mathring{e}\mathring{e}} \,\bsv
\end{equation}
for each edge $e${\AH, where $N_{\mathring{e}}$ is the number of degrees of freedom of the interior of the edge $e$}. Here, $\bsS_{e}$ 
is a Schur complement 
corresponding to the two subdomains $\Omega_i$ and $\Omega_j$ adjacent to $e$, 
and $\bsA_{\mathring{e}\mathring{e}}$ is the restriction of the global matrix $\bsA$ to the degrees of freedom on the interior of the edge. As mentioned above, Schur complement matrices correspond to energy-minimizing extensions. In this eigenvalue problem, we consider such an extension from the $e$ into the adjacent subdomains. {\AH The specific definition of $\bsS_{e}$ may vary slightly for different interface-based adaptive coarse spaces, such as in the boundary conditions in the endpoints of the edge $e$. Moreover, $\bsA_{\mathring{e}\mathring{e}}$ can also be replaced by a scaled mass matrix or a lumped version of that.}

To construct an adaptive coarse space, the eigenfunctions corresponding to eigenvalues below a user-chose tolerance $tol$ are selected and extended by zero onto the remaining interface. Then, these functions are extended in an energy minimizing way by $H_{\Gamma \rightarrow \Omega}$,
{\AH resulting in the coarse basis functions.}

{\AH As will also be discussed in~\cref{sect:Motivation} and visualized in~\cref{fig:energy_minimizing_extensions}, eigenvalue problem~\cref{eq:evp1} relates a low-energy extension ($\bsS_{e}$) and a high-energy extension ($\bsA_{\mathring{e}\mathring{e}}$). Therefore, the resulting coarse basis functions corresponding to low eigenvalues capture the critical components of the coefficient function, resulting in a condition number bound of the form~\cref{eq:cond_bound},
which is independent of the contrast of the coefficient function; see, e.g.,~\cite{Heinlein:2018:Multb,Heinlein:2019:AGD,Heinlein:2020:AGD,Gander:2015:Anala}.
}

\section{Motivation: Challenges and Advantages of Robust Algebraic Preconditioners and Key New Ideas}\label{sect:Motivation}

\lila{Despite the rapid development in adaptive coarse spaces, the development of algebraic adaptive coarse spaces, {\AH that is,} adaptive coarse spaces that {\AH can} be constructed based on the fully assembled system matrix $\bsA$, has been an open question for a longer time. For the practical applicability of a solver, this is, however, a desirable property. In particular, if the solver is fully algebraic, it can be used in any {\AH FE} implementation which provides the linear system {\AH of equations}~\cref{eq:alg} as the solver input. }

Let us briefly discuss the \lila{main challenge}, which can by understood in a similar way as in~\cref{fig:lagrangian_interpolation}. The left{\AH -} and right{\AH -}hand sides of the eigenvalue problem~\cref{eq:evp1} correspond to energies of certain {\AH extensions} of the edge values into the interior of the adjacent subdomains. In the \lila{top left} of~\cref{fig:energy_minimizing_extensions}, we see some function $u$ with an energy $\left| u \right|_{a,\Omega}^2$ that does not depend on $\alpha_{\text{max}}$; the gradient in the yellow region is zero (green). In \lila{addition three different extensions of the edge values of this function are depicted}: The extension on the top right is the extension by zero, and the corresponding energy clearly depends on $\alpha_{\text{max}}$ (red). The matrix in the right{\AH -}hand side of~\cref{eq:evp1} corresponds to this extension, that is, if $\bsu_{\mathring{e}}$ is the discrete vector with the {\AH interior} edge values 
of $u$, then $\bsu_{\mathring{e}}^T \bsA_{\mathring{e}\mathring{e}} \bsu_{\mathring{e}}$ is the energy of the extension by zero (top right in~\cref{fig:energy_minimizing_extensions}){\AH. T}his extension is algebraic since the matrix can be extracted from $\bsA$. The extension on the bottom right is the energy-minimizing extension of $u_e$ with Neumann data on the boundary; of all functions with trace $u_e$, this function has the minimum energy, which is clearly smaller or equal to the energy of $u$. Hence, the function satisfies an energy estimate of the form~\cref{eq:est_coarse_interp} with a constant $C$ which does not depend on the contrast of $\alpha$, $\frac{\alpha_{\text{max}}}{\alpha_{\text{min}}}$. This is precisely {\AH the extension employed in $\bsS_{e}$ in the left{\AH -}hand side of~\cref{eq:evp1}.} 

Unfortunately, since this extension has Neumann boundary data, the corresponding matrix $\bsS_e$ cannot be extracted algebraically from $\bsA$. An algebraic alternative to this extension would be the energy-minimizing extension with Dirichlet boundary data; cf.~\cref{fig:energy_minimizing_extensions} (bottom left). 
{\AH The}
Dirichlet submatrix corresponding to both subdomains adjacent to the edge $e$, which is required to compute this extension, can {\AH also} be extracted from $\bsA$. However, the energy of this extension clearly depends on $\alpha_{\text{max}}$. Hence, using this algebraic extension in~\cref{eq:evp1}, would also result in a dependency on $\frac{\alpha_{\text{max}}}{\alpha_{\text{min}}}$.

\begin{wrapfigure}{r}{0.65\textwidth}
	\begin{center}
\includegraphics[width=0.32\textwidth]{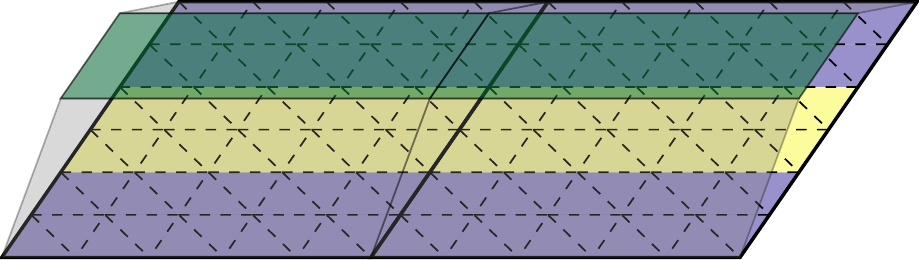} 
\includegraphics[width=0.32\textwidth]{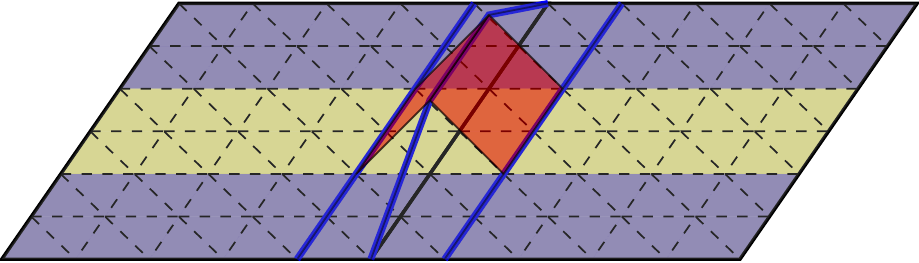} \\
\includegraphics[width=0.32\textwidth]{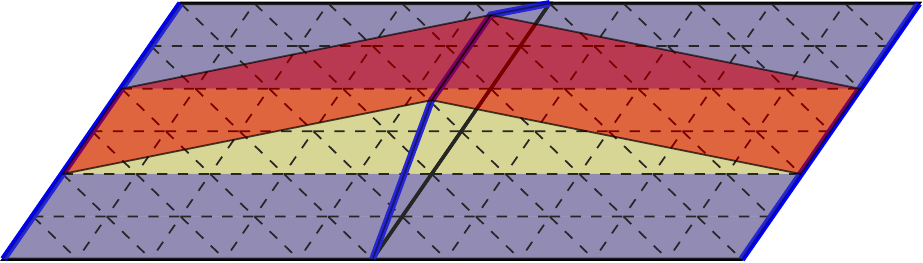}
\includegraphics[width=0.32\textwidth]{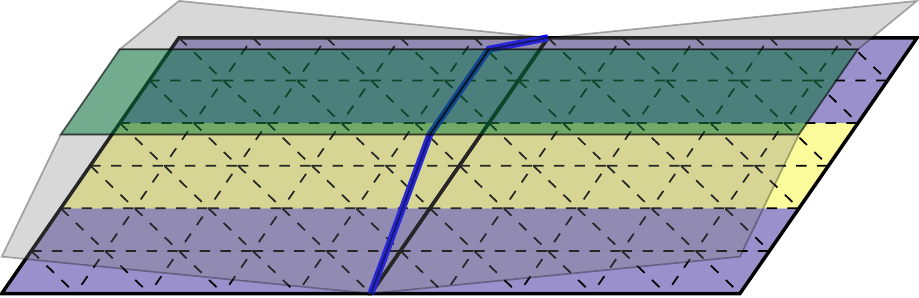}
\end{center}
\caption{For the same configuration as in~\cref{fig:lagrangian_interpolation}, we consider three different extensions of the edge values of a function (top, left): the zero extension $E_{e \to \mOe}$ (top, right), the harmonic extension with Dirichlet boundary data $H_{e \rightarrow \mOe}^{\partial\Omega_{e}}$ (bottom, left), and an harmonic extension with Neumann boundary data $H_{e \rightarrow \mOe}$ (bottom, right). The Dirichlet data of the extensions is marked in blue. Note that, in general, only the extension with Neumann boundary data has a guaranteed lower (or equal) energy compared to the original function; the energy of the other functions may even depend on the contrast. 
\label{fig:energy_minimizing_extensions}
}
\end{wrapfigure}
This short discussion explains why the Neumann matrix in~\cref{eq:evp1} cannot be simply replaced by a Dirichlet matrix; this has also been discussed in~\cite{Heinlein:2018:Multb}. It can be observed that, except for {\AH very} recent approaches targeting algebraic adaptive coarse spaces, all early adaptive coarse spaces are based on eigenvalue problems which use Neumann matrices or information about the coefficient function and geometric information. {\AH Note that there have also been attempts to heuristically construct algebraic robust coarse spaces; cf., e.g, \cite{Heinlein:2018:Multb,Heinlein:2016:POS,heinlein2020frugal}. However, no theory has been proved for these approaches yet, and they might not be robust for arbitrary coefficient functions.} 

\section{A fully algebraic and robust adaptive coarse space based on optimal local approximation spaces} \label{sect:new method}

\lila{In this section, we propose, to the best of our knowledge for the first time, a fully algebraic and robust adaptive {\AH interface}-based coarse space. To that end, we start with the vertex space and the edge space containing the constant function and enhance it by edge eigenfunctions of the \textit{Dirichlet eigenvalue problem} introduced in~\cref{subsect:Dirichlet}  and the \textit{Transfer eigenvalue problem} introduced in ~\cref{subsect:trans}. As we only use Dirichlet matrices in the eigenvalue problems to obtain an algebraic method, we {\AH might} miss the constant function in the edge space. 
However, it is well known in {\AH DDMs} that a scalable coarse has to be able to represent the null space of the global Neumann operator on each subdomain which does not touch the global Dirichlet boundary.}
\lila{We thus start with
\begin{align}
X_{vert}  &=  \text{span} \lbrace H_{\Gamma \rightarrow \Omega} ( E_{\mathcal{V} \rightarrow \Gamma} (1) ) : \mathcal{V} \text{ vertex} \rbrace, \quad \text{and}\label{def:vertex_space}\\
X_{const}  &=  \text{span} \lbrace H_{\Gamma \rightarrow \Omega} ( E_{e \rightarrow \Gamma} (1) ) : e \text{ edge} \rbrace,\label{def:constant_edge_space}
\end{align}
where we consider the vertices and edges of the nonoverlapping domain decomposition $\left\lbrace \Omega_{i} \right\rbrace_{i=1,\ldots,N}$. Here, $E_{\mathcal{V} \to \Gamma}:V_{\mathcal{V}} \to V_{\Gamma}^{0}$ and $E_{e \to \Gamma}:V_{e}^{0} \to V_{\Gamma}^{0}$ extend the FE function by zero from the vertex $\mathcal{V}$ or the edge $e$ to the interface $\Gamma$, respectively.

Note that $X_{\text{GDSW}} = X_{vert} \oplus X_{const}$ corresponds to the classical GDSW coarse space. This space is also automatically included in the AGDSW adaptive coarse space. In particular, the constant edge function corresponds to the zero eigenvalue in the AGDSW edge eigenvalue problem due to full Neumann boundary data; hence, the function is always selected as a basis function.
}

\subsection{A Dirichlet eigenvalue problem for the edges}\label{subsect:Dirichlet}

As motivated in~\cref{sect:Motivation}, for each edge $e \subset \Gamma$, we consider a Dirichlet eigenvalue problem which is a slight modification of the Neumann eigenvalue problem~\cref{eq:evp1} used in the AGDSW coarse spaces; cf.~\cite{Heinlein:2018:AGD,Heinlein:2019:AGD,Heinlein:2020:AGD}. 
First, we introduce for a fixed but arbitrary edge $e$ a so-called oversampling domain $\Omega_{e}$ such that $\dist (\partial\Omega_{e}, e) \geq \upsilon_{e} > 0$; {\AH see~\cref{fig:channels_varying_length_comb,fig:comp_evps} for illustrations of oversampling domains}. In addition, we introduce $E_{e \to \mOe}:V_{e}^{0} \to V_{\Omega_{e}}^{0}$, which assigns the coefficients of the FE functions on the edge $e$ to the FE basis functions in $\Omega_{e}$ whose associated nodes lie on $e$ and extends by zero on all other nodes in $\Omega_{e}$; see~\cref{fig:energy_minimizing_extensions} (top right) {\AH and~\cref{fig:comp_evps} (third column)} for illustrations.

We may then introduce the {\AH corresponding} inner product $b_{e}: V_{e}^{0} \times V_{e}^{0} \to \R$ defined 
\begin{align} \label{eq:c_diri}
b_e (\chi,\zeta) := a_{\Omega_{e}} ( E_{e \to \mOe}\chi, E_{e \to \mOe}\zeta) \qquad \forall \chi, \zeta \in V_{e}^{0}.
\end{align}
Furthermore, we introduce the inner product $d_{e}: V_{e}^{0} \times V_{e}^{0} \to \R$ defined as 
	\begin{align} \label{eq:d_diri}
	d_e (\chi,\zeta) := a_{\Omega_{e}} ( H_{\mathring{e} \rightarrow \mOe}^{ \partial\Omega_{e}} ( R_{e \to \mathring{e}}\chi ), H_{\mathring{e} \rightarrow \mOe}^{ \partial\Omega_{e}} (R_{e \to \mathring{e}} \zeta)) \qquad \forall \chi, \zeta \in V_{e}^{0},
	\end{align}
where $\mathring{e}$ denotes the discrete interior of the edge $e$ and, algebraically, $R_{e \to \mathring{e}}$ simply removes the degrees of freedom associated with the corners of the edge. 
Moreover, $H_{\mathring{e} \rightarrow \mOe}^{ \partial\Omega_{e}}:V_{\mathring{e}} \to V_{\Omega_{e}}^{0}$ is defined {\AH as}
	\begin{align*}
	a_{\Omega_{e}}(H_{\mathring{e} \rightarrow \mOe}^{ \partial\Omega_{e}}\chi, v) = 0 \enspace \forall v \in V_{\Omega_{e}}^{0}, \enspace H_{\mathring{e} \rightarrow \mOe}^{ \partial\Omega_{e}}\chi = 0 \text{ on } \partial\Omega_{e}, \enspace
	H_{\mathring{e} \rightarrow \mOe}^{ \partial\Omega_{e}}\chi = \chi \text{ on } \mathring{e};
	\end{align*}
hence, the upper index $\pOe$ denotes that the harmonic extension has homogeneous Dirichlet boundary data on $\pOe$ instead of Neumann boundary data{\AH; cf.~the discussion in~\cref{sect:Motivation}.}
Finally, we denote by $\| \cdot \|_{b_{e}}$ and $\| \cdot \|_{d_{e}}$ the respective norms.
We may then consider the following eigenvalue problem: find $(\psi_{e}^{(i)}, \mu^{(i)}) \in (V_{e}^{0},\mathbb{R}^{+})$ such that 
\begin{equation} \label{eq:evp_diri}
d_e (\psi_{e}^{(i)}, \chi) = \mu^{(i)} b_e (\psi_{e}^{(i)}, \chi ) \quad \forall \chi \in V_{e}^{0}.
\end{equation}
We select all $n_{dir,e}$ eigenfunctions corresponding to eigenvalues below a chosen tolerance $tol_{dir}$ and define the space 
\begin{equation}\label{def:diri coarse space}
	X_{dir}  :=  \text{span} \lbrace H_{\Gamma \rightarrow \Omega} (E_{e \to \Gamma}\psi_{e}^{(i)}): e \text{ edge}, \mu^{(i)} \leq tol_{dir} \rbrace.
\end{equation}
Here, $E_{e \to \Gamma}:V_{e}^{0} \to V_{\Gamma}^{0}$ extends the FE function by zero from $e$ to $\Gamma$, and $H_{\Gamma \rightarrow \Omega}: V_{\Gamma}^{0} \to V_{\Omega}^{0}$ is the energy-minimizing extension from the interface into the interior of the subdomains as introduced in~\cref{subsect:adaptive coarse spaces}. \lila{An eigenvalue problem similar to \cref{eq:evp_diri} has already been considered in~\cite{Heinlein:2018:Multb}, but robustness of the resulting coarse space could not be shown {\AH theoretically}.}

We emphasize that none of the operators and matrices involved in~\cref{eq:evp_diri} require any additional assembly and rely only on matrices that can be extracted from $\boldsymbol{A}$; see~\cref{sec:implementation} for details. \lila{However, as can be seen in~\cref{sec:results} in~\cref{tab:channels_varying_length}, 
choosing $X_{0}=X_{vert} \oplus X_{const} \oplus X_{dir}$
does in general not yield a robust coarse space. To obtain a {\AH (provably)} robust coarse space, we introduce a second eigenvalue problem.}

\subsection{Transfer eigenvalue problem}\label{subsect:trans}

To obtain a robust coarse space, we exploit a well-known $a$-orthogonal decomposition of $V_{\Omega_{e}}$\footnote{To simplify the notation we assume here that $\partial\Omega_{e} \cap \partial\Omega=\emptyset$; otherwise $V_{\Omega_{e}}$ has to be replaced by $V_{\Omega_{e}}^{\partial\Omega}:=\{v \in V_{\Omega_{e}}\, :\, v = 0 \text{ on } \partial\Omega\}$ \AH here and henceforth.}, namely that every $u \in V_{\Omega_{e}}$ can be written as 
\begin{align} \label{eq:split}
u = u_{\rm \Omega_{e},ha} + u_{\rm \Omega_{e},ha}^\perp,
\end{align}
where $u_{\rm \Omega_{e},ha}^\perp \in V^0_{\Omega_{e}}$, and $u_{\rm \Omega_{e},ha}$ satisfies 
\begin{equation}\label{eq:def:harmonic}
a_{\Omega_{e}}(u_{\rm \Omega_{e},ha},v) = 0 \quad \forall v \in V^0_{\Omega_{e}} \quad \text{and} \quad u_{\rm \Omega_{e},ha}|_{\pOe}=u|_{\pOe},
\end{equation}
that is, it is an energy-minimizing extension; cf.~\cref{eq:energy_min_ext}.
The decomposition~\cref{eq:split} is $a$-orthogonal, {\AH thanks to~\cref{eq:def:harmonic}}, and hence, we have
\begin{equation} \label{eq:stability}
a_{\Omega_{e}} (u_{\rm \Omega_{e},ha}, u_{\rm \Omega_{e},ha}^\perp) = 0 \quad \text{and} \quad | u |_{a,\Omega_{e}}^{2} = | u_{\rm \Omega_{e},ha} |_{a,\Omega_{e}}^{2} + |u_{\rm \Omega_{e},ha}^\perp |_{a,\Omega_{e}}^{2}.
\end{equation}

In our approach, the eigenvalue problem from \cref{subsect:Dirichlet} will serve to control the trace of functions in $V^0_{\Omega_{e}}$ on $e$. By introducing an eigenvalue problem on the space of functions that locally solve the PDE 
\begin{equation}
V_{\rm \Omega_{e},ha}:=\{ w \in V_{\Omega_{e}} \, : \, a_{\Omega_{e}}(w,v) = 0 \quad \forall v \in V_{\Omega_{e}}^{0}\},
\end{equation}
we can control the trace of functions in $V_{\rm \Omega_{e},ha}$ on $e$ as well and therefore the traces of all functions in $V_{\Omega_{e}}$ on $e$. 
The fact that the decomposition is $a$-orthogonal and {\AH the right} equation~\eqref{eq:stability} in particular will allow us to combine the contributions from both eigenvalue problems when deriving the bound for the condition number in \cref{sect:bound cond number}.

The eigenvalue problem we consider on the space of local solutions of the PDE $V_{\rm \Omega_{e},ha}$ has originally been suggested in a slightly different form in the context of multiscale and localized model order reduction methods in~\cite{SmePat16}. In that paper, it has been introduced to construct interface spaces that yield a provably optimally convergent static condensation approximation of the solution of the PDE. Similar to \cite{SmePat16} we introduce a suitable inner product $( \cdot , \cdot )_{\pOe}: V_{\pOe} \times V_{\pOe} \rightarrow \R$, where we require that the induced norm $\| \cdot \|_{\pOe}$ satisfies
\begin{equation}\label{eq:norm_equival}
c_{1} \|\zeta \|_{\pOe} \leq \sqrt{\alpha_{min}} \| \zeta \|_{L^2(\pOe)} \leq c_{2} \|\zeta \|_{\pOe} \quad \text{for all } w \in V_{\pOe},
\end{equation}
with constants $c_{1}$ and $c_{2}$ that are independent of the mesh size and the coefficient $\alpha$. 

\begin{remark}[Discussion of inner product on $\pOe$] 
	Based on the proof in~\cref{subsec:novel}, \lila{$a_{\Oe} ( H_{\pOe \rightarrow \Omega_{e}}(\cdot) , H_{\pOe \rightarrow \Omega_{e}}(\cdot))$ is  a natural choice as it leads to a very simple bound for the condition number.	 However, $a_{\Oe} (\cdot, \cdot)$ cannot not be computed algebraically as it requires the local Neumann matrices on $\Oe$. Instead, we choose
\begin{equation}\label{eq:def:inner prod bound}
	( \chi, \zeta )_{\pOe}
	=
	\alpha_{min} {\AH \frac{h}{N_{\pOe}}}
	( \chi, \zeta )_{l^2(\pOe)},
\end{equation}
which also satisfies~\cref{eq:norm_equival} thanks to $\| \cdot \|_{L^2(\pOe)}^2 \equiv \frac{h}{N_{\pOe}} \| \cdot \|_{l^2(\pOe)}^2,$ 
where $N_{\pOe}$ denotes the number of {\AH FE} nodes on $\pOe$.
Note that $\alpha_{min}$, $h$, and $N_{\pOe}$ are only constant scaling factors which can be included into $( \cdot , \cdot )_{\pOe}$ or the choice of a suitable $tol_{tr}$. {\AH If a non-uniform triangulation is employed, $h$ would have to be replaced by, for instance, the length of the edge $\left| e \right|$ in the equivalence of the $L^2$- and $l^2$-norms.}
We remark that even though the algebraic choice \cref{eq:def:inner prod bound} yields a slightly more complicated bound of the condition number (see~\cref{subsec:novel}), a numerical comparison in \cref{sec:results} shows that the coarse spaces corresponding to \cref{eq:def:inner prod bound} and $a_{\Oe} ( H_{\pOe \rightarrow \Omega_{e}}(\cdot) , H_{\pOe \rightarrow \Omega_{e}}(\cdot))$ perform very similar{\AH ly; see also~\cref{fig:comp_evps}}.}
\end{remark}

Next, as in \cite{SmePat16}, we 
introduce the transfer operator $\widetilde{T}: V_{\pOe} \rightarrow \{w|_{e} \, , w \in V_{\rm \Omega_{e},ha}\}=:V_{\rm \Omega_{e},ha,e}$\footnote{\lila{Again for $\partial\Omega_{e}\cap \partial\Omega \neq \emptyset$ one needs to replace $V_{\pOe}$ by $V_{\pOe}^{\partial\Omega}:=\{v \in V_{\pOe}\, : \, v = 0 \text{ on } \partial\Omega\}.$}} defined as 
$\widetilde{T}\varsigma:=(H_{\pOe \rightarrow \Omega_{e}}\varsigma)|_{e}$ for $\varsigma \in V_{\pOe}$,
where the harmonic extension operator $H_{\pOe \rightarrow \Omega_{e}}:V_{\pOe} \to V_{\Omega_{e}}$ is defined 
{\AH as}
\begin{equation*}
a_{\Omega_{e}}(H_{\pOe \rightarrow \Omega_{e}}\chi, v) = 0 \enspace \forall v \in V_{\Omega_{e}}^{0}, \quad \text{and} \quad H_{\pOe \rightarrow \Omega_{e}}\chi = \chi \text{ on } \partial\Omega_{e}; 
\end{equation*}
{\AH cf.~\cref{subsect:adaptive coarse spaces}.}
However, as $X_{vert}$, introduced in \eqref{def:vertex_space}, already accounts for the degrees of freedom in the vertices of the coarse discretization, we wish to only take into account the behavior of the functions in the interior of $e$. Therefore, we define the slightly modified transfer operator $T=(I - I_{\mathcal{V},e})\widetilde{T}$, where $I_{\mathcal{V},e}$ denotes the restriction of $I_{\mathcal{V}}$ to $e$, and consider the following transfer eigenvalue problem: Find $(\phi^{(i)},\lambda_{i}) \in (V_{\rm \Omega_{e},ha},\R^{+})$ such that
\begin{equation}\label{eq:transfer_eigenproblem}
b_e (T(\phi^{(i)}_{e}|_{\pOe}), T(w|_{\pOe})) = \lambda^{(i)} (\phi^{(i)}_{e}|_{\pOe}, w|_{\pOe} )_{\pOe}\quad \forall w \in V_{\rm \Omega_{e},ha}.
\end{equation}
We select the eigenfunctions corresponding to the $n_{tr,e}$ largest eigenvalues such that {\AH $\lambda^{(n_{tr,e})} > tol_{tr}$ and $\lambda^{(n_{tr,e}+1)} \leq tol_{tr}$} and introduce 
\begin{align}
\varphi^{(i)}_{e}&=T(\phi^{(i)}_{e}|_{\pOe}), \enspace i=1,\hdots \qquad \text{and} \label{def:transfer_basis}\\
X_{tr} & =  \text{span} \lbrace H_{\Gamma \rightarrow \Omega} (E_{e \to \Gamma}(\varphi^{(i)}_{e})): e \text{ edge}, \lambda^{(i)} > tol_{tr} \rbrace.\label{def:trans coarse space}
\end{align}
\lila{We highlight that thanks to~\cref{eq:def:inner prod bound} the calculation of the eigenfunctions of \cref{eq:transfer_eigenproblem} only requires access to the global stiffness matrix $\bsA$ resulting in a fully algebraic coarse space; for further details and the computational realization of \cref{eq:transfer_eigenproblem}, see \cref{sec:implementation}.}

Note that we may also perform a singular value decomposition (SVD) of the operator $T$, which reads 
$T\zeta=\sum_{i}\sigma^{(i)}\hat{\varphi}^{(i)}(\chi^{(i)},\zeta)_{\pOe}$ for $\zeta \in V_{\pOe}$ with orthonormal bases $\hat{\varphi}^{(i)} \in V_{\rm \Omega_{e},ha,e}^{0}:=\{(w|_{e}-I_{\mathcal{V},e}(w|_e)) \, , w \in V_{\rm \Omega_{e},ha}\}$, $\chi^{(i)} \in V_{\pOe}$, and singular values $\sigma^{(i)} \in \mathbb{R}^{+}_{0}$. Then, we have, up to numerical errors, $\sigma^{(i)} = \sqrt{\lambda^{(i)}}$ and
$\spanlin \{\varphi^{(1)}_{e},\hdots,\varphi^{(n)}_{e}\} = \spanlin \{\hat{\varphi}^{(1)},\hdots,\hat{\varphi}^{(n)}\}$.  We can thus infer from results of the optimality of the SVD that the space $\Lambda^{n}=\spanlin \{\varphi^{(1)}_{e},\hdots,\varphi^{(n)}_{e}\}$ minimizes $\|T - \Pi_{\Lambda^{n}}T\|$ 
among all $n$-dimensional subspaces of $V_{\rm \Omega_{e},ha}^{0}$ and hence optimally (in the sense of Kolmogorov) approximates the range of $T$ and thus $V_{\rm \Omega_{e},ha,e}^{0}$; see also \cite{BabLip11,Pin85,SmePat16}.

\subsection{\AH Complete definition of the adaptive coarse space} \label{sec:fa_coarse_space}

Finally, we define the complete adaptive coarse space $X_0$ as
\begin{equation}\label{def:adaptive_coarse_space}
X_{0}
=
X_{\text{VCDT}}
:=
X_{vert} \oplus  X_{const} \oplus X_{dir} \oplus X_{tr},
\end{equation}
where $X_{vert}$, $X_{const}$, $X_{dir}$, and $X_{tr}$ were defined in \cref{def:vertex_space}, \cref{def:constant_edge_space}, \cref{def:diri coarse space}, and \cref{def:trans coarse space}, respectively.

\section{Bound of the condition number}\label{sect:bound cond number}

As we use exact local {\AH and coarse} solvers, \cref{eq:def local bil form,eq:def:bil form 0}, we obtain local stability {\AH with stability constant $\omega = 1$}; cf.~\cite[Assumption~2.4]{Toselli:2005:DDM} and{\AH, for the readers' convenience,} Assumption SM1.1. Thanks to \cite[Lemma~3.11 and follow-up discussion]{Toselli:2005:DDM} and the proof of \cite[Theorem~4.1]{Dryja:1994:DAS}, we obtain
\begin{align}  \label{eq:cond_est}
\kappa\left( \bsM_{AS}^{-1} \bsA \right) 
\leq
C_0^2 \left( m+1 \right),
\end{align}
where $C_0^2$ is the constant in the stable decomposition in \cref{ass:stabledecomp} below~\cite[Assumption~2.2]{Toselli:2005:DDM}, and $m \in \mathbb{N}$ is an upper bound for the number of overlapping subdomains $\Omega_{i}^{\prime}$ any {\AH FE node} in $\Omega$ can belong to. Hence, $m$ only depends on the structure of the overlapping domain decomposition and is, for regular domain decompositions, bounded from above; also for domain decompositions generated by mesh partitioners such as METIS~\cite{Karypis:1998:ASP}, $m$ is generally reasonably small.
Therefore, in order to obtain a condition number bound, it is sufficient to bound the constant $C_0^2$ in the stable decomposition:

\begin{assumption}[Stable decomposition {\cite[Assumption 2.2]{Toselli:2005:DDM}}]\label{ass:stabledecomp}
	There exists a constant $C_0$, such that every $u \in V_{\Omega}^{0}$ admits a decomposition
	\begin{align}
	u &= E_{0}u_{0} + \sum_{i=1}^M E_{\mOip\to \Omega} u_i  \label{eq:decomp}\\
	\text{such that}\qquad		&\sum_{i=0}^M a_{\pOmega_{i}}(u_{i},u_{i}) \leq C_0^2 a_{\Omega}(u,u).\label{eq:stabledecomp}
	\end{align}
\end{assumption}

\subsection{Construction of the stable decomposition} 
\label{subsec:defdecomp}
One key novelty of the present manuscript{\AH, which is a key ingredient for proving robustness,} 
lies in the {\AH specific} definition of the coarse component $u_{0}$. To that end, we define projection operators 
\begin{align}
\Pi_{e,dir} v &:= \sum_{i=1}^{n_{dir,e}} b_e ( v , \psi_{e}^{(i)} ) \psi_{e}^{(i)} \quad \forall v \in V^{0}_{e} \quad \text{and}\label{def:proj op diri} \\
\Pi_{e,tr} v &:= \sum_{i=1}^{n_{tr,e}} \frac{1}{\lambda^{(i)}} b_e ( v , \varphi^{(i)}_{e}) \varphi^{(i)}_{e}\quad \forall v \in V_{\rm \Omega_{e},ha,e}^{0}, \label{def:proj op trans}
\end{align}
where $n_{dir,e}$ and $n_{tr,e}$ denote the number of selected eigenfunctions from the eigenproblems \cref{eq:evp_diri} and \cref{eq:transfer_eigenproblem}, respectively. Note that, for all $v \in V_{\rm \Omega_{e},ha,e}^{0}$ and $v = T(\widetilde{v}|_{\pOe})$, we have
$
\Pi_{e,tr} v = \sum_{i=1}^{n_{tr,e}} ( \widetilde{v}|_{\pOe} , \phi^{(i)}|_{\pOe} )_{\pOe} \varphi^{(i)}_{e}
$, where $\phi^{(i)}$ is the $i$-th eigenfunction of the transfer eigenvalue problem \cref{eq:transfer_eigenproblem}. Exploiting the latter and the definition $\varphi^{(i)}_{e}:=T(\phi^{(i)}_{e}|_{\pOe})$ (see \cref{def:transfer_basis}) yields the expression of $\Pi_{e,tr}v$ in \cref{def:proj op trans}. Moreover, fitting the definition of the coarse spaces $X_{dir}$ and $X_{tr}$ in \cref{def:diri coarse space} and \cref{def:trans coarse space}, respectively, we also introduce the maps \lila{
	\begin{align*}
	\Pi_{dir} v &:=  H_{\Gamma \rightarrow \mO} ( E_{\me \rightarrow \Gamma} (\Pi_{e,dir} v )) \quad\forall v \in V_{e}^{0}, \quad \text{and} \\
	\Pi_{tr} v &:=  H_{\Gamma \rightarrow \mO} ( E_{\me \rightarrow \Gamma} (\Pi_{e,tr} v)) \quad\forall v \in V_{\rm \Omega_{e},ha,e}^{0}.
	\end{align*}}

\paragraph*{\textbf{Definition of the coarse interpolation}}

As we will later invoke Poincar\'{e}'s inequality {\AH in the proof of the condition number bound}, 
we have to include the constant function in our interpolation{\AH. In this step, we 
	have to introduce the minimum value of the coefficient function $\alpha_{\min}$ into the estimate, which is why we already included it in the right{\AH -}hand side of the transfer eigenvalue problem; see~\cref{eq:norm_equival}. Our final condition number bound will still be robust since the maximum eigenvalue and, hence, the contrast will not appear.}
Similarly as in \cite{SmePat16}, we exploit that functions which are constant on $\Omega_{e}$ lie in $V_{\rm \Omega_{e},ha}$. We may thus write 
\begin{align}\label{eq:splitting_u}
u|_{\Omega_{e}} = \widehat{u_{\rm \Omega_{e},ha}} + c_{u}\mathbbm{1}_{\Omega_{e}} + u_{\rm \Omega_{e},ha}^\perp, \text{ where }\enspace  \widehat{u_{\rm \Omega_{e},ha}}:=u_{\rm \Omega_{e},ha} - c_{u}\mathbbm{1}_{\Omega_{e}} \in V_{\rm \Omega_{e},ha},
\end{align}
\lila{and $c_{u} \in \R$ is zero when $\pOe \cap \partial\Omega \neq \emptyset$ and will be selected later otherwise. }We define the coarse interpolation \lila{
\begin{align}
	u_0 &:=H_{\Gamma \rightarrow \mO}((I_{\mathcal{V}}u)|_{\Gamma}) + \sum\limits_{e \in \Gamma}  (u_{0,\rm \Omega_{e},ha} +  u_{0,\rm \Omega_{e},ha}^\perp), \quad \text{where}\label{eq:def:coarse_interp}    \\
	\nonumber	u_{0,\rm \Omega_{e},ha} &:= \Pi_{tr} (\widehat{u_{\rm \Omega_{e},ha}}|_e - I_{\mathcal{V},e}\widehat{u_{\rm \Omega_{e},ha}}|_e) + H_{\Gamma \rightarrow \mO} ( E_{\me \rightarrow \Gamma} ( c_{u}\mathbbm{1}_{e} - I_{\mathcal{V},e}c_{u}\mathbbm{1}_{e}))  \\
	\nonumber u_{0,\rm \Omega_{e},ha}^\perp &:= \Pi_{dir} (u_{\rm \Omega_{e},ha}^\perp|_{e} - I_{\mathcal{V},e}u_{\rm \Omega_{e},ha}^\perp|_{e}),
\end{align}}
where $I_{\mathcal{V}}$ is the interpolation by the nodal functions in $X_{vert}$ corresponding to the vertices, and $I_{\mathcal{V},e}$ is the {\AH restriction of this} 
interpolation 
to the edge $e$. To simplify notations, we assume that $u_{0}$ is already expressed in the basis of $V_{\Omega}^{0}$.

\paragraph*{\textbf{Definition of the local components $\boldsymbol{u_{i}}$, $\boldsymbol{i=1,\hdots,M}$}}
{\AH As typical in the proof of condition number estimates for Schwarz methods, w}e define 
\begin{equation}\label{eq:def:local_contr}
u_i = 
{\AH R_{\Omega \to \mOip}}
I^h(\theta_i(u-u_0)), \quad i=1,\hdots,M,
\end{equation}
where $I^{h}(\theta_i(u-u_0)) \in V_{\Omega}^{0}$ denotes the interpolant of $\theta_i(u-u_0)$ and $\lbrace\theta_i\rbrace_{i=1}^M$ is a partition of unity corresponding to an overlapping decomposition ${\AH \lbrace \widetilde{\Omega}_{i} \rbrace_{i=1}^M}$ of the domain $\Omega$ with overlap $h${\AH; hence, it is also a partition of unity on the overlapping domain decomposition ${\AH \lbrace \mOip \rbrace_{i=1}^M}$.} In detail, we require for a {\AH FE} node $x_h$ that $\theta_i(x_h)=1$ in $\Omega_i$, $\theta_i(x_h)=1/m_{x_h}$ on $\Gamma$ and $\theta_i(x_h)=0$ otherwise; here $m_{x_h}$ denotes the number of subdomains $\Omega_{i}$ the {\AH FE} node $x_h$ belongs to. This definition ensures \eqref{eq:decomp}. {\AH See, for example, \cite[Section~3.6]{Toselli:2005:DDM} for the same construction with a slightly different choice of the partition of unity. 
}

\subsection{Bound of the coarse level and local contributions}\label{subsec:bound coarse local}

{\AH The general structure of the proof follows earlier works, such as~\cite{Gander:2015:Anala,Heinlein:2018:Multb,Heinlein:2019:AGD,Heinlein:2020:AGD}. We first prove estimates for the coarse and local components, as summarized in~\cref{lem:coarse level local bound}, and then combine them to the final estimate for $C_0^2$ in~\cref{prop:total bound const stable decomp} in~\cref{subsec:complete bound}. 
Plugged into~\cref{eq:cond_est}, we obtain the final condition number estimate.
The main step in the proof is finding an upper bound for the term $| E_{\me \rightarrow \mOe} \left[(u - u_0)|_{e}\right] |_{a,\Omega_{e}}^2$. This corresponds to upper bounds for $| w_{e_{ij}} (u - u_0)|_{a,\Omega_{e}}^2$ in OS-ACMS~\cite{Heinlein:2018:Multb} or $| z_{\xi \to \Omega_\xi} (u - u_0)|_{a,\Omega_\xi}^2$ in AGDSW~\cite{Heinlein:2018:AGD}, respectively; for details, see the corresponding references. The proof of this bound differs significantly for our new method and is the central novel contribution of this manuscript in terms of numerical analysis and the topic of the next subsection. The remainder of the proof of the bound of the coarse level and local contributions, as summarized in~\cref{lem:coarse level local bound}, is standard. It follows the earlier works listed above, and can, for the readers' convenience, be found in subsection SM1.1.
}

\begin{lemma}[Bound of the coarse level and local contributions]\label{lem:coarse level local bound}
	Let $u_{0}$ and $u_{i}$ be defined as in \cref{eq:def:coarse_interp} and \cref{eq:def:local_contr}, respectively, and let $m_{e}$ denote the maximal number of edges $e$ in a subdomain $\Omega_{i}$. Then, we have
	\begin{align}
	| u_0 |_{a,\Omega}^2 &\leq 2 | u |_{a,\Omega}^2 + 2 m_{e} \sum_{i=1}^{M} \sum_{e \subset \partial\Omega_i} \left| E_{\me \rightarrow \mOe} \left[(u - u_0)|_{e}\right] \right|_{a,\Omega_{e}}^2, \label{eq:bound coarse level}\\
	\sum_{i=1}^{M} | u_i |_{a,\pOmega_{i}}^2 &\leq 18 | u |_{a,\Omega}^2 + 15 m_{e} \sum_{i=1}^{M} \sum_{e \subset \partial\Omega_i} \left| E_{\me \rightarrow \mOe} \left[(u - u_0)|_{e}\right] \right|_{a,\Omega_{e}}^2,\label{eq:bound local contr}
	\end{align}
	where $E_{\me \to \mOe}$ has been defined in the beginning of \cref{subsect:Dirichlet}.
\end{lemma}

\subsection{\lila{Bound of the extension of $\boldsymbol{(u - u_{0})|_{e}}$}}\label{subsec:novel}
To complete the proof of the bound of the condition number it thus remains to estimate the term $| E_{\me \rightarrow \mOe} \left[(u - u_0)|_{e}\right] |_{a,\Omega_{e}}^2$, which is the topic of this subsection. 
{\AH 
Different from other approaches, we decompose $V_{\Omega_{e}}$ in an $a$-orthogonal way and derive a bound for each part separately. 
Since the decomposition is $a$-orthogonal, we are finally able to combine both parts again. In particular, a}s already indicated in \cref{subsect:trans} and more specifically in \cref{eq:split}, we exploit that every $u \in V_{\Omega_{e}}$ can be written as 
\begin{align} \label{eq:split2}
u = u_{\rm \Omega_{e},ha} + u_{\rm \Omega_{e},ha}^\perp,
\end{align}
where $u_{\rm \Omega_{e},ha}^\perp \in V^0_{\Omega_{e}}$ and  $u_{\rm \Omega_{e},ha}$ satisfies 
\begin{equation}\label{eq:def:harmonic2}
a_{\Omega_{e}}(u_{\rm \Omega_{e},ha},v) = 0 \quad \forall v \in V^0_{\Omega_{e}}, \quad \text{and} \quad u_{\rm \Omega_{e},ha}|_{\pOe}=u|_{\pOe}.
\end{equation}
Thanks to the definition of the coarse level contribution in \cref{eq:def:coarse_interp}, \lila{$u_{0,\rm \Omega_{e},ha}|_{e^{*}}=0$ and $u_{0,\rm \Omega_{e},ha}^\perp|_{e^{*}}=0$ for $e \neq e^{*}$}, and the linearity of the interpolation operator $I_{\mathcal{V}}$, we may then split the term $| E_{\me \rightarrow \mOe} \left[(u - u_0)|_{e}\right] |_{a,\Omega_{e}}^2$ as follows
\begin{align}
	| E_{\me \rightarrow \mOe} \left[(u - u_0)|_{e}\right] |_{a,\Omega_{e}}^2 & \leq 2\Bigl| E_{\me \rightarrow \mOe} \left[\left( u_{\rm \Omega_{e},ha} -  I_{\mathcal{V}}u_{\rm \Omega_{e},ha} - u_{0,\rm \Omega_{e},ha} \right)|_{e}\right] \Bigr|_{a,\Omega_{e}}^2 \label{eq:special_est_aux1}\\
	&\quad\enspace + 2\Bigl|E_{\me \rightarrow \mOe} \left[\left(u_{\rm \Omega_{e},ha}^\perp -  I_{\mathcal{V}}u_{\rm \Omega_{e},ha}^\perp - u_{0,\rm \Omega_{e},ha}^\perp \right)|_{e}\right] \Bigr|_{a,\Omega_{e}}^2. \label{eq:special_est_aux2}
\end{align}
Next, by exploiting the properties of the eigenvalue problems introduced in \cref{subsect:Dirichlet,subsect:trans}, we will estimate the terms in \cref{eq:special_est_aux1} and \cref{eq:special_est_aux2} separately, and combine the estimate at the end by exploiting {\AH $a$-orthogonality, that is,}
\begin{equation}\label{eq:stability2}
| u |_{a,\Omega_{e}}^{2} = | u_{\rm \Omega_{e},ha} |_{a,\Omega_{e}}^{2} + |u_{\rm \Omega_{e},ha}^\perp |_{a,\Omega_{e}}^{2}.
\end{equation}

\paragraph*{\textbf{Bound of the harmonic part}} To estimate the terms containing the harmonic part of $u$, we will make use of the fact that the space $\Lambda_{n}=\spanlin \{\varphi^{(1)}_{e},\hdots,\varphi^{(n)}_{e}\}$ minimizes $\|T - \Pi_{e,tr}T\|$ among all $n$-dimensional subspaces of $V_{\rm \Omega_{e},ha}^{0}$ (see last paragraph of \cref{subsect:trans} for {\AH the} definition) and that for all $v \in V_{\rm \Omega_{e},ha,e}=\{w|_{e} \, , w \in V_{\rm \Omega_{e},ha}\}$  we have
\begin{equation} \label{eq:evp1:spectral2_2}
\left\| \left(v|_e - I_{\mathcal{V},e}(v|_e)\right)  - \Pi_{e,tr} \left(v|_e - I_{\mathcal{V},e}(v|_e)\right)  \right\|_{b_{e}} \leq \sqrt{tol_{tr}} \| v|_{\pOe} \|_{\pOe}.
\end{equation} 
This result follows from the well-known fact that $\|T - \Pi_{e,tr}T\|=\sigma_{n_{tr,e}+1}$ as $\Lambda_{n}$ is the span of the first $n$ left singular vectors of $T$; cf.~the Eckart-Young theorem, for instance, in~\cite{GolVan13} for the discrete setting and Theorem 2.2 in Chapter 4 in~\cite{Pin85} for infinite-dimensional spaces, and~\cite{BabLip11, SmePat16} for how to use this result to derive error bounds for the local and global approximation error of optimally converging multiscale methods. We may then show the following result.

\begin{proposition}[Bound of the harmonic part]\label{lem:bound_harm}
	Let the constants $c_{t,e}$ and $c_{p,e}$ be defined as 
	\lila{\begin{equation}\label{eq:def:constants}
		c_{t,e}:=\sup_{v \in V_{\rm \Omega_{e},ha}} \frac{\| v |_{\partial\Omega_{e}} \|_{L^{2}(\pOe)}}{\|v\|_{H^{1}(\Omega_{e})}}, 
		\quad
		c_{p,e}:=\sup_{v \in V_{\rm \Omega_{e},ha}} \frac{\|v - \frac{c_{\Omega_{e}}}{|\Omega_{e}|} \int_{\Omega_{e}} v \mathbbm{1}_{\Omega_{e}}  \|_{L^{2}(\Omega_{e})}}{\|\nabla v \|_{L^{2}(\Omega_{e})}},
		\end{equation}
		where $c_{\Omega_{e}}=1$ if $\pOe \cap \partial\Omega = \emptyset$ and $0$ otherwise.} Then, we have
	\begin{equation}\label{eqbound_ha}
		\Bigl| E_{\me \rightarrow \mOe} \left[\left( u_{\rm \Omega_{e},ha} -  I_{\mathcal{V}}u_{\rm \Omega_{e},ha} - u_{0,\rm \Omega_{e},ha} \right)|_{e}\right] \Bigr|_{a,\Omega_{e}}^2  \leq \frac{c_{t,e}^{2}tol_{tr} (1 + c_{p,e}^{2}) }{c_{1}^{2}}| u_{\rm \Omega_{e},ha} |_{a,\Omega_{e}}^{2}, 
	\end{equation}
	{\AH where $c_1$ depends on the choice of the bilinear form $( \cdot, \cdot )_{\pOe}$ in~\cref{eq:def:inner prod bound}; see~\cref{eq:norm_equival}.}
\end{proposition}
\begin{proof}
	By exploiting the definition of $u_{0,\rm \Omega_{e},ha}$ in \cref{eq:def:coarse_interp}, the restriction to the edge $e$, the linearity of the interpolation operators $I_{\mathcal{V}}$ and $I_{\mathcal{V},e}$, the definition of $\widehat{u_{\rm \Omega_{e},ha}}$ in \cref{eq:splitting_u}, and the definition of the inner product $b_{e}$ in \cref{eq:c_diri}, we obtain
	\begin{align*}
	&\Bigl| E_{\me \rightarrow \mOe} \Bigl[\Bigl( u_{\rm \Omega_{e},ha} -  I_{\mathcal{V}}u_{\rm \Omega_{e},ha} - u_{0,\rm \Omega_{e},ha} \Bigr)\Bigr|_{e}\Bigr] \Bigr|_{a,\Omega_{e}}^2 \\
	&\enspace = \Bigl| E_{\me \rightarrow \mOe} \Bigl[\Bigl( u_{\rm \Omega_{e},ha} -  I_{\mathcal{V}}u_{\rm \Omega_{e},ha} - \Pi_{tr} (\widehat{u_{\rm \Omega_{e},ha}}|_e - I_{\mathcal{V},e}\widehat{u_{\rm \Omega_{e},ha}}|_e) \\
	&\qquad\qquad\qquad - H_{\Gamma \rightarrow \mO} ( E_{\me \rightarrow \Gamma} ( c_{u}\mathbbm{1}_{e} - I_{\mathcal{V},e}c_{u}\mathbbm{1}_{e}))
	\Bigr)\Bigr|_{e}\Bigr] \Bigr|_{a,\Omega_{e}}^2 \\
	&\enspace = \Bigl| E_{\me \rightarrow \mOe} \Bigl[ (u_{\rm \Omega_{e},ha} |_{e} -  I_{\mathcal{V},e}u_{\rm \Omega_{e},ha}|_{e}) - \Pi_{e,tr} (\widehat{u_{\rm \Omega_{e},ha}}|_e - I_{\mathcal{V},e}\widehat{u_{\rm \Omega_{e},ha}}|_e) \\
	&\qquad\qquad\qquad - ( c_{u}\mathbbm{1}_{e} - I_{\mathcal{V},e}c_{u}\mathbbm{1}_{e})
	\Bigr] \Bigr|_{a,\Omega_{e}}^2 \\
	&\enspace = \Bigl| E_{\me \rightarrow \mOe} \Bigl[ (\widehat{u_{\rm \Omega_{e},ha}}|_e - I_{\mathcal{V},e}\widehat{u_{\rm \Omega_{e},ha}}|_e) - \Pi_{e,tr} (\widehat{u_{\rm \Omega_{e},ha}}|_e - I_{\mathcal{V},e}\widehat{u_{\rm \Omega_{e},ha}}|_e)\Bigr] \Bigr|_{a,\Omega_{e}}^2 \\
	&\enspace = \bigl\| (\widehat{u_{\rm \Omega_{e},ha}}|_e - I_{\mathcal{V},e}\widehat{u_{\rm \Omega_{e},ha}}|_e) - \Pi_{e,tr} (\widehat{u_{\rm \Omega_{e},ha}}|_e - I_{\mathcal{V},e}\widehat{u_{\rm \Omega_{e},ha}}|_e) \bigr\|_{b_e}^2.
	\end{align*}
	As $\widehat{u_{\rm \Omega_{e},ha}}|_e$ is in $V_{\rm \Omega_{e},ha,e}=\{w|_{e} \, , w \in V_{\rm \Omega_{e},ha}\}$, we may then invoke \cref{eq:evp1:spectral2_2} {\AH and obtain} 
	\begin{equation}\label{eq:est special aux1}
	\Bigl| E_{\me \rightarrow \mOe} \Bigl[\Bigl( u_{\rm \Omega_{e},ha} -  I_{\mathcal{V}}u_{\rm \Omega_{e},ha} - u_{0,\rm \Omega_{e},ha} \Bigr)\Bigr|_{e}\Bigr] \Bigr|_{a,\Omega_{e}}^2 \leq tol_{tr} \| \widehat{u_{\rm \Omega_{e},ha}}|_{\pOe} \|_{\pOe}^{2}. 
	\end{equation}

	To conclude the estimate of the harmonic part of the function, we exploit \eqref{eq:norm_equival} {\AH and} choose $c_{u}$ in $\widehat{u_{\rm \Omega_{e},ha}}:=u_{\rm \Omega_{e},ha} - c_{u}\mathbbm{1}_{\Omega_{e}}$ as $c_{u}=(1/|\Omega_{e}|) \int_{\Omega_{e}} u_{\rm \Omega_{e},ha}$ \lila{if $\pOe\cap\partial\Omega = \emptyset$ and $c_{u}=0$ otherwise}. {\AH Now, we} apply the trace theorem and the Poincar\'{e} inequality:
	\begin{align*}
	&\Bigl| E_{\me \rightarrow \mOe} \Bigl[\Bigl( u_{\rm \Omega_{e},ha} -  I_{\mathcal{V}}u_{\rm \Omega_{e},ha} - u_{0,\rm \Omega_{e},ha} \Bigr)\Bigr|_{e}\Bigr] \Bigr|_{a,\Omega_{e}}^2 \\
	&\quad \leq tol_{tr} \| \widehat{u_{\rm \Omega_{e},ha}}|_{\pOe} \|_{\pOe}^{2} \leq tol_{tr} \frac{\alpha_{min}}{c_{1}^{2}} \|\widehat{u_{\rm \Omega_{e},ha}}|_{\pOe} \|_{L^{2}(\pOe)}^{2} \\
	&\quad \leq tol_{tr} \frac{c_{t,e}^{2}\alpha_{min}}{c_{1}^{2}} (\|\widehat{u_{\rm \Omega_{e},ha}} \|_{L^{2}(\Omega_{e})}^{2} + \|\nabla \widehat{u_{\rm \Omega_{e},ha}}  \|_{L^{2}(\Omega_{e})}^{2}) \\
	&\quad \leq tol_{tr} \frac{c_{t,e}^{2}\alpha_{min}}{c_{1}^{2}} (c_{p,e}^{2}\|\nabla u_{\rm \Omega_{e},ha} \|_{L^{2}(\Omega_{e})}^{2} + \|\nabla \widehat{u_{\rm \Omega_{e},ha}}  \|_{L^{2}(\Omega_{e})}^{2}) \\
	&\quad = tol_{tr} \frac{c_{t,e}^{2}\alpha_{min}}{c_{1}^{2}}  (1 + c_{p,e}^{2}) \|\nabla u_{\rm \Omega_{e},ha} \|_{L^{2}(\Omega_{e})}^{2}  \leq  \frac{c_{t,e}^{2}tol_{tr}}{c_{1}^{2}} (1 + c_{p,e}^{2}) | u_{\rm \Omega_{e},ha} |_{a,\Omega_{e}}^{2}. 
	\end{align*}
\end{proof}

\paragraph*{\textbf{Bound of the perpendicular part}}
By exploiting that the eigenfunctions $\psi_{e}^{(i)}$, $i=1,\hdots,\dim (V_{e}^{0})$, of \cref{eq:evp_diri} span $V_{e}^{0}$ and that we select all $n_{dir,e}$ eigenfunctions corresponding to eigenvalues below a chosen tolerance $tol_{dir}$ to define the space $X_{dir}$, we obtain {\AH by standard spectral arguments for adaptive coarse spaces} that
\begin{equation} \label{eq:evp2:spectral1}
\| v \|_{d_e}^2 = \| \Pi_{e,dir} v \|_{d_e}^2 + \| v - \Pi_{e,dir} v \|_{d_e}^2, 
\  
\| v - \Pi_{e,dir} v \|_{b_e}^2 \leq \frac{1}{tol_{dir}} \| v - \Pi_{e,dir} v \|_{d_e}^2,
\end{equation}
{\AH for each $v \in V_e^0$; cf.~, e.g.,~\cite{KRR15a,KRR15b,Spillane:2014:Absta,Heinlein:2018:Multb,Heinlein:2019:AGD}.} Using these two results, we may prove the following result.

\begin{proposition}[Bound of the perpendicular part]\label{lem:bound_perp}
	We have 
	\begin{equation}\label{eq:bound_perp}
	\Bigl|E_{\me \rightarrow \mOe} \left[\left(u_{\rm \Omega_{e},ha}^\perp -  I_{\mathcal{V}}u_{\rm \Omega_{e},ha}^\perp - u_{0,\rm \Omega_{e},ha}^\perp \right)|_{e}\right] \Bigr|_{a,\Omega_{e}}^2 \leq \frac{1}{tol_{dir}} |u_{\rm \Omega_{e},ha}^\perp|_{a,\Omega_{e}}^{2}.
	\end{equation}
\end{proposition}
\begin{proof}
	By exploiting the definition of $u_{0,\rm \Omega_{e},ha}^\perp$ in \cref{eq:def:coarse_interp}, the restriction to the edge $e$, the linearity of the interpolation operators $I_{\mathcal{V}}$ and $I_{\mathcal{V},e}$, the definition of the inner product $b_{e}$ in \cref{eq:c_diri}, and \cref{eq:evp2:spectral1}, we obtain
	\begin{align*}
	&\Bigl|E_{\me \rightarrow \mOe} \left[\left(u_{\rm \Omega_{e},ha}^\perp -  I_{\mathcal{V}}u_{\rm \Omega_{e},ha}^\perp - u_{0,\rm \Omega_{e},ha}^\perp \right)|_{e}\right] \Bigr|_{a,\Omega_{e}}^2 \\
	& = \Bigl|E_{\me \rightarrow \mOe} \left[\left(u_{\rm \Omega_{e},ha}^\perp -  I_{\mathcal{V}}u_{\rm \Omega_{e},ha}^\perp - \Pi_{dir} (u_{\rm \Omega_{e},ha}^\perp|_{e} - I_{\mathcal{V},e}u_{\rm \Omega_{e},ha}^\perp|_{e}) \right)|_{e}\right] \Bigr|_{a,\Omega_{e}}^2 \\
	& = \Bigl|E_{\me \rightarrow \mOe} \left[\left((u_{\rm \Omega_{e},ha}^\perp|_{e} -  I_{\mathcal{V},e}u_{\rm \Omega_{e},ha}^\perp|_{e}) - \Pi_{e,dir} (u_{\rm \Omega_{e},ha}^\perp|_{e} - I_{\mathcal{V},e}u_{\rm \Omega_{e},ha}^\perp|_{e}) \right)\right] \Bigr|_{a,\Omega_{e}}^2 \\
	& = \|(u_{\rm \Omega_{e},ha}^\perp|_{e} -  I_{\mathcal{V},e}u_{\rm \Omega_{e},ha}^\perp|_{e}) - \Pi_{e,dir} (u_{\rm \Omega_{e},ha}^\perp|_{e} - I_{\mathcal{V},e}u_{\rm \Omega_{e},ha}^\perp|_{e}) \|_{b_{e}}^2 \\
	& \overset{\cref{eq:evp2:spectral1}}{\leq}\frac{1}{tol_{dir}}  \|u_{\rm \Omega_{e},ha}^\perp|_{e} -  I_{\mathcal{V},e}u_{\rm \Omega_{e},ha}^\perp|_{e} \|_{d_{e}}^2.
	\end{align*}
	A close inspection of the definition of the inner product $d_{e}$ in~\cref{eq:d_diri} reveals that, as we have $R_{e \to \mathring{e}} (u_{\rm \Omega_{e},ha}^\perp|_{e} -  I_{\mathcal{V},e}u_{\rm \Omega_{e},ha}^\perp|_{e}) = R_{e \to \mathring{e}} (u_{\rm \Omega_{e},ha}^\perp|_{e})$,
	there holds 
	\begin{align*}
	\|u_{\rm \Omega_{e},ha}^\perp|_{e} -  I_{\mathcal{V},e}u_{\rm \Omega_{e},ha}^\perp|_{e} \|_{d_{e}}^2= a_{\Omega_{e}} ( H_{\mathring{e} \rightarrow \mOe}^{ \partial\Omega_{e}} R_{e \to \mathring{e}}(u_{\rm \Omega_{e},ha}^\perp|_{e}), H_{\mathring{e} \rightarrow \mOe}^{ \partial\Omega_{e}}R_{e \to \mathring{e}} (u_{\rm \Omega_{e},ha}^\perp|_{e})).
	\end{align*}
	As $u_{\rm \Omega_{e},ha}^\perp \in V_{\Omega_e}^{0}$ and the discrete harmonic extension $H_{\mathring{e} \rightarrow \mOe}^{ \partial\Omega_{e}} R_{e \to \mathring{e}}(u_{\rm \Omega_{e},ha}^\perp|_{e})$ minimizes the $| \cdot |_{a,\Omega_e}$-norm among all functions in $V_{\Omega_e}^{0}$ that equal $u_{\rm \Omega_{e},ha}^\perp$ on $\mathring{e}$, we conclude that
	\begin{equation*}
	\|u_{\rm \Omega_{e},ha}^\perp|_{e} -  I_{\mathcal{V},e}u_{\rm \Omega_{e},ha}^\perp|_{e} \|_{d_{e}}^2 \leq |u_{\rm \Omega_{e},ha}^\perp|_{a,\Omega_{e}}^{2}.
	\end{equation*}
\end{proof}

\paragraph*{\textbf{Combining the bounds of the harmonic and perpendicular parts}}
By invoking the stability result \cref{eq:stability2} and exploiting the estimate \cref{eq:special_est_aux1}-\cref{eq:special_est_aux2}, \cref{lem:bound_harm}, and \cref{lem:bound_perp}, we obtain the following result.
\begin{corollary}\label{coro:total_special_bound}
	We have
	\begin{equation*}
	| E_{\me \rightarrow \mOe} \left[(u - u_0)|_{e}\right] |_{a,\Omega_{e}}^2 \leq 2 \max\left\{\frac{c_{t,e}^{2}tol_{tr} (1 + c_{p,e}^{2}) }{c_{1}^{2}}, \frac{1}{tol_{dir}}\right\} |u|_{a,\Omega_{e}}^{2}.
	\end{equation*}
\end{corollary}

\subsection{Complete bound of {\AH the} condition number}\label{subsec:complete bound}

By combining \cref{lem:coarse level local bound} and \cref{coro:total_special_bound}, we obtain the following bound for the condition number:
\begin{proposition}  \label{prop:total bound const stable decomp} 
	Let $u_{0}$ and $u_{i}$, $i=1,\hdots,M$ be defined as in \cref{eq:def:coarse_interp} and \cref{eq:def:local_contr}, respectively. Let further $m_{e}$ denote the maximal number of edges $e$ in a subdomain $\Omega_{i}$, $\nu:=\max_{e \in \Gamma}\{$number of subdomains $\Omega_{i}$ that satisfy $\Omega_{i} \cap \mathring{e} \neq \emptyset\}$, $\widetilde{\omega}= \max_{i=1,\hdots,M}\{$number of $\Omega_{e}$ such that $\Omega_{e}\cap \Omega_{i}\neq \emptyset\}$, $c_{t}:=\max_{e \in \Gamma} c_{t,e}$, and $c_{p}:=\max_{e \in \Gamma} c_{p,e}$, where $c_{t,e}$ and $c_{p,e}$ have been defined in \cref{eq:def:constants}. Then, the condition number can be bounded as $\kappa\left( \bsM_{AS}^{-1} \bsA \right) \leq C_{0}^2 \left( m+1 \right)$ with 
	\begin{align}\label{eq:total bound const stable decomp} 
		C_{0}^2 := \left(20 + 34m_{e}\nu\widetilde{\omega} \max\left\{\frac{c_{t}^{2}tol_{tr} (1 + c_{p}^{2}) }{c_{1}^{2}}, \frac{1}{tol_{dir}}\right\}\right)
	\end{align}
	in the stable decomposition in \cref{ass:stabledecomp}. 
\end{proposition}

\section{Computational Realization} \label{sec:implementation}

\lila{In this section, we} briefly discuss the algorithmic steps for the algebraic construction of the different components of the two level Schwarz preconditioner~\cref{eq:two-level_schwarz} with the adaptive coarse space $X_{0}$ \cref{def:adaptive_coarse_space}. \lila{As $X_{\text{GDSW}} = X_{vert} \oplus X_{const} {\AH \subset X_0}$ (see discussion at the beginning of \cref{sect:new method}), many algorithmic building blocks are the same as in the standard GDSW preconditioner. We thus mainly focus here on the Dirichlet and transfer eigenvalue problems and refer to \cite{Heinlein:2016:PIT,Heinlein:2021:FAT} for details on the algebraic construction of GDSW preconditioners.}

\paragraph{The main algorithmic components of the two-level preconditioner~\cref{eq:two-level_schwarz}} 

\lila{Since $\bsE_{\mOip\to \Omega}=\bsR_{\Omega \to \mOip}^T$, $\bsE_{0}=\bsR_{0}^T$ and thus
$
\bsA_{\pOmega_{i}} = \bsR_{\Omega \to \mOip}\bsA \bsR_{\Omega \to \mOip}^T
$, and $
\bsA_{0} = \bsR_{0}\bsA \bsR_{0}^T,
$
 it is sufficient to construct the operators $\bsR_{\Omega \to \mOip}$, for {\AH$i=1,\ldots,M$, and $\bsR_{0}$.} }

\paragraph{Nonoverlapping domain decomposition} Let us assume that a nonoverlapping domain decomposition $\bar{\Omega} = \cup_i \bar{\Omega}_i$ is already given or can be obtained from the sparsity pattern of the matrix $\bsA$ using a graph partitioner, such as METIS~\cite{Karypis;1998;FHQ}. 

\paragraph{Restriction operators on the first level} \lila{For the first level, the operators $\bsR_{\Omega \to \mOip}$ extract the subvector corresponding to $\mOip$ when applied to a global {\AH FE} vector.}
Similarly, $\bsA_{\pOmega_{i}} = \bsR_{\Omega \to \mOip}\bsA \bsR_{\Omega \to \mOip}^T$ can be obtained by extracting the submatrix corresponding to $\mOip$ from $\bsA$. Hence, the $\bsR_{\Omega \to \mOip}$ never have to {\AH be} set up explicitly.

To determine the action of the  operators $\bsR_{\Omega \to \mOip}$, it is sufficient to identify the index sets of the subdomains $\mOip$. Starting from the nonoverlapping subdomains $\Oi$, the overlapping subdomains can be constructed recursively by adding layers of {\AH FE} nodes. 
This can, again, be performed based on the sparsity pattern of $\bsA$, making use of the fact that two {\AH FE} nodes with indices $k$ and $l$ share a nonzero {\AH off-diagonal} coefficient in $\bsA$ if they are adjacent.

\paragraph{Computation of the coarse basis functions} 
\lila{The coarse basis functions in terms of the {\AH FE} basis functions are stored in the columns of $\bsE_{0}$. 
If we have computed the interface values $\bsE_{0,\Gamma}$ of the coarse basis functions, the interior values are then computed as energy-minimizing extensions from the interface to the interior of the subdomains
$
\bsE_{0,I}
=
- \bsA_{II}^{-1} \bsA_{I \Gamma} \bsE_{0,\Gamma};
$
}
cf.~\cref{eq:harm_ext,eq:energy_min_ext} and the discussion in~\cref{subsect:adaptive coarse spaces}.

The steps above depend on a partition of the nodes into interior and interface nodes. This can be performed based on the multiplicity of the nodes in the nonoverlapping domain decomposition: {\AH FE} nodes which belong to only one subdomain are the interior nodes, whereas nodes which belong to two subdomains are the edge nodes, and nodes which belong to more than two subdomains are the vertices. 
Then, the interface nodes are the union of the edge and vertex nodes. Hence, the nodes can be categorized algebraically, just based on the nonoverlapping domain decomposition. 

\subsection{Algebraic construction of coarse edge functions from Dirichlet and transfer eigenvalue problem}

\paragraph{Construction of $\Oe$} Both eigenvalue problems require the oversampling domain $\Oe$ corresponding to each edge $e$. It can be constructed in a similar way as the overlapping subdomains: we start with the {\AH FE} nodes of interior of the edge $e$ and extend the set recursively layer-by-layer of elements; cf.~\cref{fig:channels_varying_length_comb}. As before, this can be done based on the sparsity pattern of $\bsA$. 

\paragraph{Dirichlet eigenvalue problem~\cref{eq:evp_diri}} The Dirichlet eigenvalue problems can be written as{\AH: Find $\left(\bsv,\mu\right) \in \mathbb{R}^{N_{\mathring{e}}} \times \mathbb{R}^+$, such that
$$
\bsS_{e} \, \bsv 
= 
\mu 
\bsA_{\mathring{e}\mathring{e}} \, \bsv 
$$
}
with matrices $\bsS_{e}$ and $\bsA_{\mathring{e}\mathring{e}}$; see also~\cref{eq:evp1}. The latter can easily be extracted from $\bsA$; it is the submatrix corresponding to the interior edge nodes. The Schur complement on the left{\AH -}hand side is given by
$
\bsS_{e} = \bsA_{\mathring{e}\mathring{e}} - \bsA_{\mathring{e}\widetilde{R}} \bsA_{\widetilde{R}\widetilde{R}}^{-1} \bsA_{\widetilde{R}\mathring{e}},
$
where the index set $\widetilde{R}$ corresponds to the interior nodes of $\Oe$ except for the interior nodes of $e$. As $\bsA_{\mathring{e}\mathring{e}}$, the matrices $\bsA_{\mathring{e}\widetilde{R}}$, $\bsA_{\widetilde{R}\widetilde{R}}$, and $\bsA_{\widetilde{R}\mathring{e}}$ can be extracted as submatrices of $\bsA$.

\paragraph{Transfer eigenvalue problem~\cref{eq:transfer_eigenproblem}} The transfer eigenvalue problem can be written in matrix form as{\AH: 
Find $\left(\bsv,\lambda \right) \in \mathbb{R}^{N_{\pOe}} \times \mathbb{R}^+$, such that	
$$
\bsT^T \, \bsA_{\mathring{e}\mathring{e}} \, \bsT \, \bsv 
= 
\lambda 
\frac{\alpha_{\min}}{N_{\pOe}}
\bsI_{\pOe} \bsv 
$$
}
where $\bsT$ is the matrix corresponding to the transfer operator, $N_{\pOe}$ is the number of {\AH FE} nodes on $\pOe${\AH, and $\bsI_{\pOe} \in \mathbb{R}^{N_{\pOe} \times N_{\pOe}}$ is the identity matrix on the degrees of freedom of $\pOe$.} 
Therefore, let 
$$
\widehat{\bsT}
= 
\begin{pmatrix}
- \bsA_{\mathring{\Omega}_{e}\mathring{\Omega}_{e}}^{-1} \bsA_{\mathring{\Omega}_{e} \pOe} \\
\bsI_{\pOe}
\end{pmatrix},
$$
where $\mathring{\Omega}_{e}$ and $\pOe$ correspond to the interior nodes of $\Oe$ and the nodes on $\pOe$, respectively. {\AH This corresponds to the energy-minimizing extension from $\pOe$ to $\Oe$; cf.~\cref{eq:harm_ext}.}
Then, 
$
\bsT
=
\bsR_{\Oe \to \mathring{e}}
\widehat{\bsT},
$
where $\mathring{e}$ denotes the discrete interior of $e$.
Again, all matrices involved in the transfer eigenvalue problem, $\bsA_{\mathring{e}\mathring{e}}$, $\bsA_{\mathring{\Omega}_{e}\mathring{\Omega}_{e}}$, and $\bsA_{\mathring{\Omega}_{e} \pOe}$, can be extracted from $\bsA$, and $\bsR_{\Oe \to \mathring{e}}$ only requires the index sets of $\Oe$ and $\mathring{e}$.

\begin{remark}
	\lila{We can efficiently approximate the space spanned by the leading singular vectors of the transfer operator $T$ by using randomization as suggested in~\cite{BuhSme18} in the context of localized model reduction. 
	This is helpful since the transfer eigenvalue problem, which is defined on the nodes of $\pOe$, may be significantly larger than the Dirichlet eigenvalue problem, which is defined on the interior nodes of $e$.}
\end{remark}

\paragraph{Orthogonaliztion of the edge functions} Let us remark that 
$
V^0_{\Omega_{e}}|_e 
\cap
V_{\rm \Omega_{e},ha}|_e
\neq 
\emptyset.
$
Hence, even though the two spaces are $a$-orthogonal on $\Oe$, when restricted to an edge, the functions from the two spaces may not be linearly independent anymore. In order to remove (almost) linearly dependent edge functions, we finally orthogonalize the edge functions for each edge using a proper orthogonal decomposition (POD) \cite{BerHol93,Sir87}.

\section{Numerical results}
\label{sec:results}

In this section, we present numerical results \lila{demonstrating} the robustness of our new adaptive coarse space in~\cref{def:adaptive_coarse_space}. In particular, we consider the heterogeneous model problem~\cref{eq:var} with coefficient function $\alpha$ introduced in~\cref{sect:problem} on the computational domain $\Omega = [0,1]^2$. We use piecewise linear finite elements on a regular mesh to discretize and obtain~\cref{eq:alg}. 
Moreover, we use the preconditioned conjugate gradient (PCG) method 
and stop the iteration once $\| \bsM^{-1}r^{(k)} \| / \| \bsM^{-1}r^{(0)} \| < 10^{-10}$. The overlapping domain decomposition into $\{\Omega^{\prime}_{i}\}_{i=1}^{M}$ is done in a structured way, resulting in square subdomains.  
\lila{
	{\AH In all numerical experiments,}
we will use an algebraic overlap of $1h$, that is, we extend the nonoverlapping subdomains by one layer of finite element nodes to obtain the overlapping subdomains. {\AH Moreover, we keep the threshold for the selection of eigenfunctions in the Dirichlet eigenvalue problem $tol_{dir} = 10^{-3}$ and the tolerance for the POD orthogonalization $tol_{o} = 10^{-5}$ fixed. The threshold for the transfer eigenvalue problem $tol_{tr}$ is chosen as $10^{5}$ in most cases and only varied in a few cases as reported in the tables.} The algorithms have been implemented and run using MATLAB\_R2021a.}

We compare the adaptive coarse space proposed in this paper in~\cref{def:adaptive_coarse_space} and variants of it with the classical two-level GDSW coarse space and 
AGDSW
coarse spaces. Even though our theory holds for general coefficient functions, we are mostly interested in testing our coarse space for difficult configurations, where standard coarse spaces fail. 
Let us briefly comment on the design of the coefficient functions. We concentrate on discontinuous coefficient functions instead of continuously varying coefficient functions since large coefficient jumps at the domain decomposition interface have a stronger influence on the convergence; this is also reflected by the results in~\cref{sec:results:spe10}. Furthermore, it is well-known that coefficient jumps inside the subdomains have only a minor influence; see, for example~\cite{Gippert:2012;AFE}. 
It has also been observed that those examples where the high coefficient components do not touch the Dirichlet boundary of the domain are the more difficult ones; this behavior is particularly evident for low numbers of subdomains. 
Therefore, except for 
the realistic coefficient function in~\cref{sec:results:spe10}, we set one layer of elements with low coefficient at the boundary of the domain.

\subsection{A first model problem\lila{: Channels of varying lengths}} \label{sec:results:first}

\begin{figure}
\begin{center}
\includegraphics[height=0.3\textwidth]{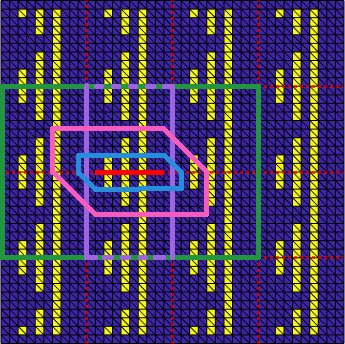}
\hspace*{0.05\textwidth}
\includegraphics[height=0.3\textwidth]{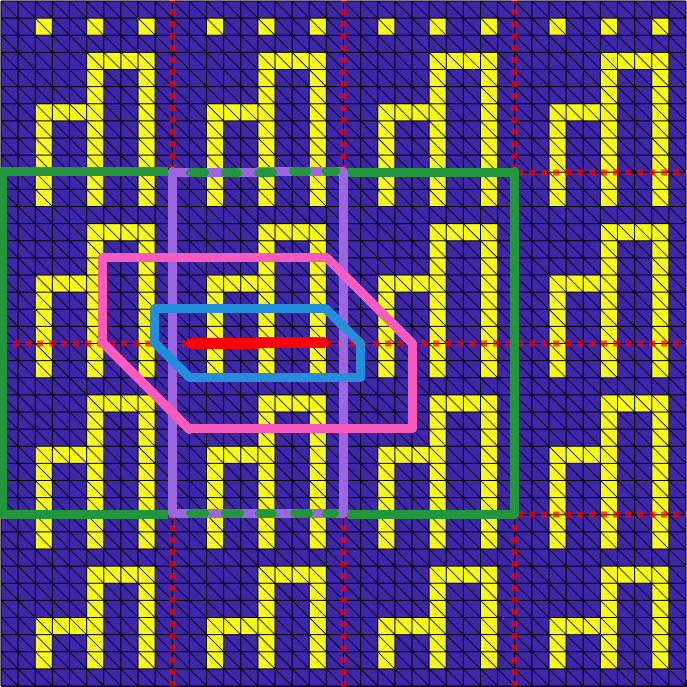}
\end{center}
\caption{Heterogeneous coefficient functions with high-coefficient channels of varying lengths (\textbf{left}) and comb-type high-coefficient components (\textbf{right}) cutting the edges on a $4 \times 4$ subdomains with $H/h = 10$ and different choices of oversampling domains $\Omega_{e}$: the domain decomposition interface is depicted as dashed red lines, the elements with $\alpha = \alpha_{\max}$ are colored yellow, and the elements with $\alpha = \alpha_{\min}$ are colored dark blue. In addition, the oversampling domain $\Omega_e^{2h}$ is depicted in light blue, $\Omega_e^{5h}$ in pink, $\Omega_e^{H}$ in light green, and the domain for the extensions in AGDSW eigenvalue problems in light purple; the  discrete interior edge $\mathring{e}$ is plotted in solid red. 
\label{fig:channels_varying_length_comb}}
\end{figure}

\begin{table}
\begin{center}
\begin{tabular}{|l||l@{\hspace{6pt}}l||r@{\hspace{2pt}}c@{\hspace{2pt}}r|r@{\hspace{6pt}}r|}
	\hline
	$X_0$                                  & $\Omega_{\rm out}$               & $tol_{tr}$   & \multicolumn{3}{c|}{$\dim X_0$} & \multicolumn{1}{c}{$\kappa$} & \multicolumn{1}{c|}{\# its.} \\ \hline\hline
	$X_{\text{GDSW}}$                      & --                               & --           & $33$ & / &                 $33$ &             $2.7 \cdot 10^5$ &                        $118$ \\ \hline
	$X_{\text{AGDSW}}$                     & --                               & --           & $57$ & / &                 $57$ &                        $7.4$ &                         $24$ \\ \hline
	\multirow{3}{*}{$X_{\text{VCD}}$}      & $\Omega_e^{2h}$                  & --           & $33$ & / &                 $33$ &             $2.7 \cdot 10^5$ &                        $118$ \\
	                                       & $\Omega_e^{5h}$                  & --           & $57$ & / &                 $57$ &                        $7.2$ &                         $24$ \\
	                                       & $\Omega_e^H$                     & --           & $57$ & / &                 $57$ &                        $7.2$ &                         $24$ \\ \hline
	\multirow{3}{*}{$X_{\text{VCT}-l^2}$}  & $\Omega_e^{2h}$                  & $10^{5}$     & $93$ & / &                $105$ &                        $7.6$ &                         $24$ \\
	                                       & $\Omega_e^{5h}$                  & $10^{5}$     & $57$ & / &                 $66$ &                       $19.0$ &                         $36$ \\
	                                       & $\Omega_e^H$                     & $10^{5}$     & $57$ & / &                 $66$ &                       $19.0$ &                         $36$ \\ \hline
	\multirow{4}{*}{$X_{\text{VCDT}-l^2}$} & \multirow{2}{*}{$\Omega_e^{2h}$} & $10^{5}$     & $93$ & / &                $105$ &                        $7.6$ &                         $24$ \\
	                                       &                                  & $\bf 10^{6}$ & $57$ & / &                 $69$ &                        $7.6$ &                         $24$ \\
	                                       & $\Omega_e^{5h}$                  & $10^{5}$     & $57$ & / &                 $90$ &                        $7.2$ &                         $25$ \\
	                                       & $\Omega_e^H$                     & $10^{5}$     & $57$ & / &                 $90$ &                        $7.2$ &                         $24$ \\ \hline
	\multirow{3}{*}{$X_{\text{VCDT}-a}$}   & $\Omega_e^{2h}$                  & $10^{5}$     & $57$ & / &                 $69$ &                        $7.5$ &                         $24$ \\
	                                       & $\Omega_e^{5h}$                  & $10^{5}$     & $57$ & / &                 $72$ &                        $7.2$ &                         $24$ \\
	                                       & $\Omega_e^H$                     & $10^{5}$     & $57$ & / &                 $69$ &                        $7.3$ &                         $24$ \\ \hline
\end{tabular}
\end{center}
\caption{Numerical results for the coefficient function shown in~\cref{fig:channels_varying_length_comb} (left), $4 \times 4$ subdomains with $H/h = 10$, and $\alpha_{\max} = 10^6$ and $\alpha_{\min} = 1$ using different coarse spaces; in case of the novel coarse spaces $X_{\text{VCDT}-*}$, we vary the size of $\Omega_e$: two layers of finite elements ($\Omega_e^{2h}$), five layers of finite elements ($\Omega_e^{5h}$), or one layer of subdomains ($\Omega_e^{H}$) around $e$; see also~\cref{fig:channels_varying_length_comb} (left); for the $\Omega_e^{2h}$ case, we vary the threshold $tol_{tr}$. We report the coarse space dimension (final dimension / dimension before POD orthogonalization), the estimated condition number, and the iteration count. Non-default tolerances $tol_{tr}$ are marked in \textbf{bold face}.
\label{tab:channels_varying_length}}
\end{table}

As a first model problem, we consider the coefficient function shown in~\cref{fig:channels_varying_length_comb} (left){\AH .} 
The results are listed in~\cref{tab:channels_varying_length}, \lila{where here and elsewhere the reported condition number is an estimate from the Lanczos process within PCG.} They clearly indicate that the condition number for the classical coarse space is not robust with respect to the contrast of the coefficient function; in fact, the condition number of $2.7 \cdot 10^5$ is close to the contrast $\alpha_{\max}/\alpha_{\min} = 10^6$ itself. Due to the moderate number of subdomains, the resulting iteration count is still {\AH moderate, that is,} $118$.

Both the AGDW and the new coarse space are robust, resulting in a condition number below $10$ and an iteration count of $24$ or $25$. This is the case for both the fully algebraic variant ($X_{\text{VCDT}-l^2}$) and the variant that uses the $a$ bilinear form ($X_{\text{VCDT}-a}$). \lila{The latter two}
yield comparable results, also when varying the size of $\Omega_e$ from two layers of elements ($\Omega_e^{2h}$), to five layers of elements ($\Omega_e^{5h}$) and one layer of subdomains ($\Omega_e^{H}$) around the edge $e$; cf.~\cref{fig:channels_varying_length_comb} (left) for plots of the different choices of $\Omega_e$.

In~\cref{tab:channels_varying_length}, we also provide results for only using one of the two eigenvalue problems, that is either only the Dirichlet eigenvalue problem ($X_{\text{VCD}}$) or only the transfer eigenvalue problem ($X_{\text{VCT}-l^2}$). We observe that, once $\Oe$ gets too small ($\Oe^{2h}$), the Dirichlet eigenvalue problem fails to detect the high-coefficient channels. The resulting coarse space just corresponds to the standard GDSW coarse space. On the other hand, for this example, the transfer eigenvalue problem alone already yields a robust coarse space.

\begin{figure}
\centering
\includegraphics[width=0.32\textwidth]{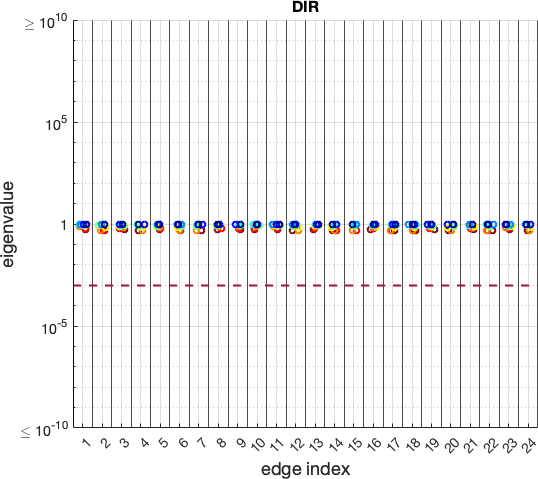}
\includegraphics[width=0.32\textwidth]{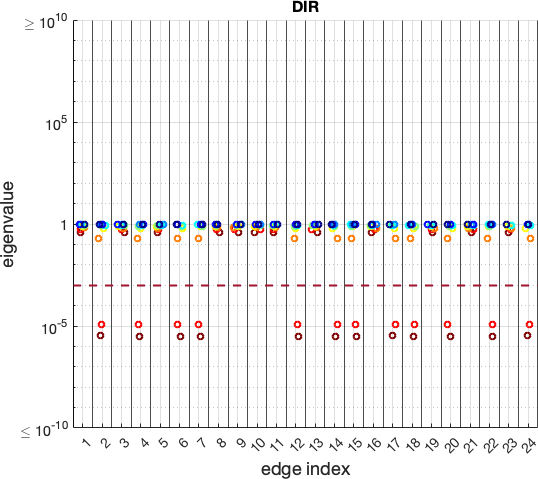}
\includegraphics[width=0.32\textwidth]{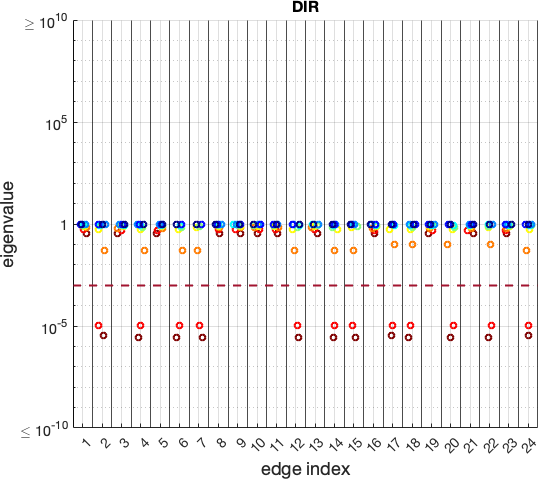}

\includegraphics[width=0.32\textwidth]{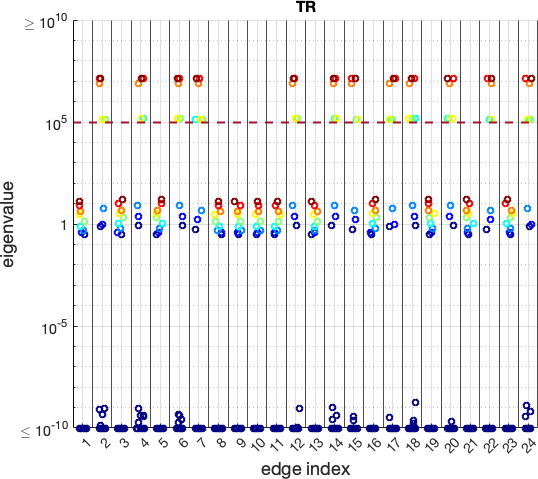}
\includegraphics[width=0.32\textwidth]{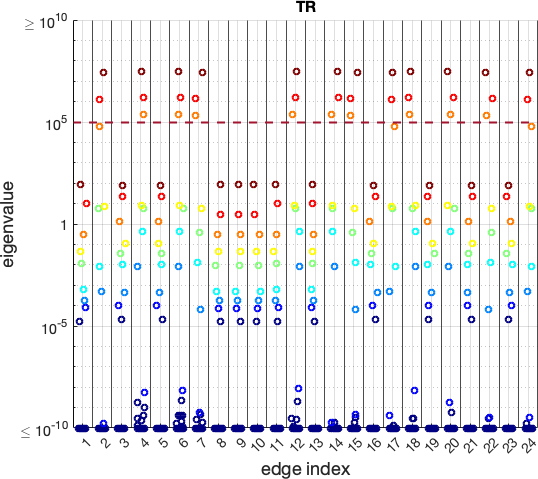}
\includegraphics[width=0.32\textwidth]{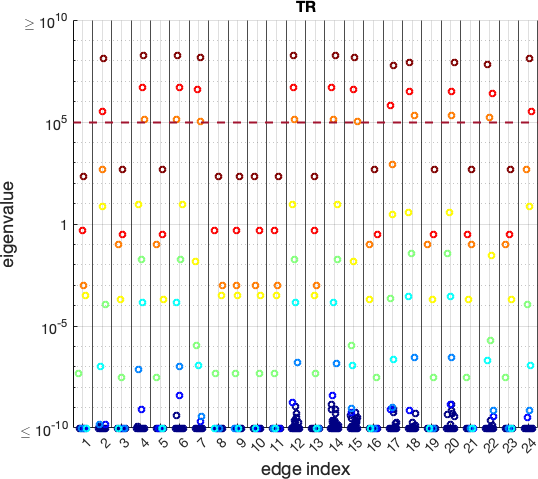}
\caption{Eigenvalue distributions for the Dirichlet (top) and transfer (bottom) eigenvalue problems for configuration in~\cref{fig:channels_varying_length_comb} (left) for different sizes of $\Omega_e$: $\Omega_e^{2h}$ (left), $\Omega_e^{5h}$ (middle), and $\Omega_e^H$ (right); see also~\cref{fig:channels_varying_length_comb} (left). Threshold of $10^{-3}$ and $10^{-5}$ for the Dirichlet and transfer eigenvalue problems are plotted as dashed red lines. Note that all eigenvalues below $10^{-10}$ or larger than $10^{10}$ are plotted as $10^{-10}$ or $10^{10}$, respectively. The eigenvalues are colored based on their ordering, from the largest eigenvalue (red) to the smallest eigenvalue (blue).
\label{fig:ev_distribution_omega_e}}
\end{figure}

\Cref{fig:ev_distribution_omega_e} shows the eigenvalue distribution for the Dirichlet and transfer eigenvalue problems of $X_{\text{VCDT}-l^2}$ for varying the size of $\Omega_e$. We observe that, for a smaller $\Omega_e$, the eigenvalues are much more concentrated in certain parts of the spectrum. For a larger choice of $\Omega_e$, they are are more distributed across the spectrum. The \lila{eigenmodes necessary for obtaining robustness} can be identified by the gap in the spectrum, which should contain the threshold. Choosing a larger threshold for the Dirichlet eigenvalue problem and smaller threshold for the transfer eigenvalue problem, respectively, increases the chance of catching all relevant eigenfunctions, however, it may result in a larger coarse space dimension than necessary. 

The case of $\Omega_e^{2h}$ shows that the tolerance has to be increased from $10^5$ to $10^6$ in order to omit those eigenmodes which result in a too large coarse space dimension; this is also reflected in the results in~\cref{tab:channels_varying_length}.

Furthermore, we observe that there is a significant amount of linearly dependent edge basis functions for the new coarse space; this results from the fact that \lila{we combine the constant function and eigenfunctions from the Dirichlet and transfer eigenvalue problem}. Due to POD orthogonalization the total coarse space dimension is reduced by $12$--$15$. {\AH Here, the resulting coarse space dimension of $57$ is always optimal, which can be explained as follows:
For each edge, we need at least one (constant) function, and in case of channels cutting the edge, at least as many functions as channels. This yields $12 + 3 \times 12 = 48$ edge functions. In addition to that, we obtain one function for each of the $9$ vertices, resulting in a dimension of $57$.}

\begin{table}
\begin{center}
\begin{tabular}{|l|l@{\hspace{6pt}}l||r@{\hspace{2pt}}c@{\hspace{2pt}}r|r@{\hspace{6pt}}r|}
	\hline
	$\alpha_{\min}$            & $X_0$                 & $tol_{tr}$ & \multicolumn{3}{c|}{$\dim X_0$} & \multicolumn{1}{c}{$\kappa$} & \multicolumn{1}{c|}{\# its.} \\ \hline\hline
	\multirow{2}{*}{$10^{-2}$} & $X_{\text{GDSW}}$     & --         & $33$ &  /   &              $33$ &             $2.7 \cdot 10^7$ &                        $142$ \\
	                           & $X_{\text{VCDT}-l^2}$ & $10^{4}$   & $57$ &  /   &              $93$ &                        $7.3$ &                         $25$ \\ \hline
	\multirow{2}{*}{$1$}       & $X_{\text{GDSW}}$     & --         & $33$ &  /   &              $33$ &             $2.7 \cdot 10^5$ &                        $118$ \\
	                           & $X_{\text{VCDT}-l^2}$ & $10^{4}$   & $57$ &  /   &              $93$ &                        $7.2$ &                         $25$ \\ \hline
	\multirow{2}{*}{$10^{2}$}  & $X_{\text{GDSW}}$     & --         & $33$ &  /   &              $33$ &             $2.7 \cdot 10^3$ &                         $95$ \\
	                           & $X_{\text{VCDT}-l^2}$ & $10^{4}$   & $57$ &  /   &              $69$ &                        $8.5$ &                         $25$ \\ \hline
\end{tabular}
\end{center}
\caption{Numerical results for the coefficient function shown in~\cref{fig:channels_varying_length_comb} (left) with varying $\alpha_{\min}$, $4 \times 4$ subdomains with $H/h = 10$, and $\alpha_{\max} = 10^6$ and $\alpha_{\min} = 1$ using the classical $X_{\text{GDSW}}$ coarse space and the adaptive $X_{\text{VCDT}-l^2}$ coarse space for $\Omega_e = \Omega_e^{5h}$; see also~\cref{fig:channels_varying_length_comb} (left). We report the coarse space dimension (final dimension / dimension before POD orthogonalization), the estimated condition number, and the iteration count.
\label{tab:channels_varying_length:alpha_min}}
\end{table}

Finally, we report result{\AH s} for varying values of $\alpha_{\min}$ in~\cref{tab:channels_varying_length:alpha_min}. We observe that, for the classical GDSW coarse space, condition number and iteration count deteriorate; in particular, the condition number is in the order of the coefficient contrast. For the $X_{\text{VCDT}-l^2}$ coarse space, the results are robust and independent of $\alpha_{\min}$.

\subsection{Algebraic and non-algebraic variants of the transfer eigenvalue problem} 

\begin{figure}
\centering
\begin{minipage}{0.24\textwidth} \centering
\includegraphics[width=\textwidth,height=\textwidth]{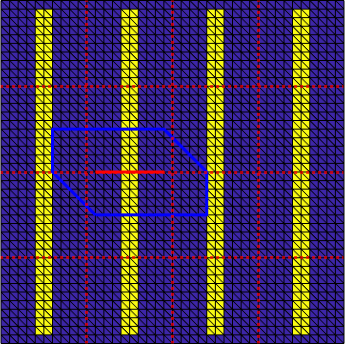}
\end{minipage}
\begin{minipage}{0.24\textwidth} \centering
\includegraphics[width=\textwidth,trim=35mm 89mm 13mm 61mm,clip]{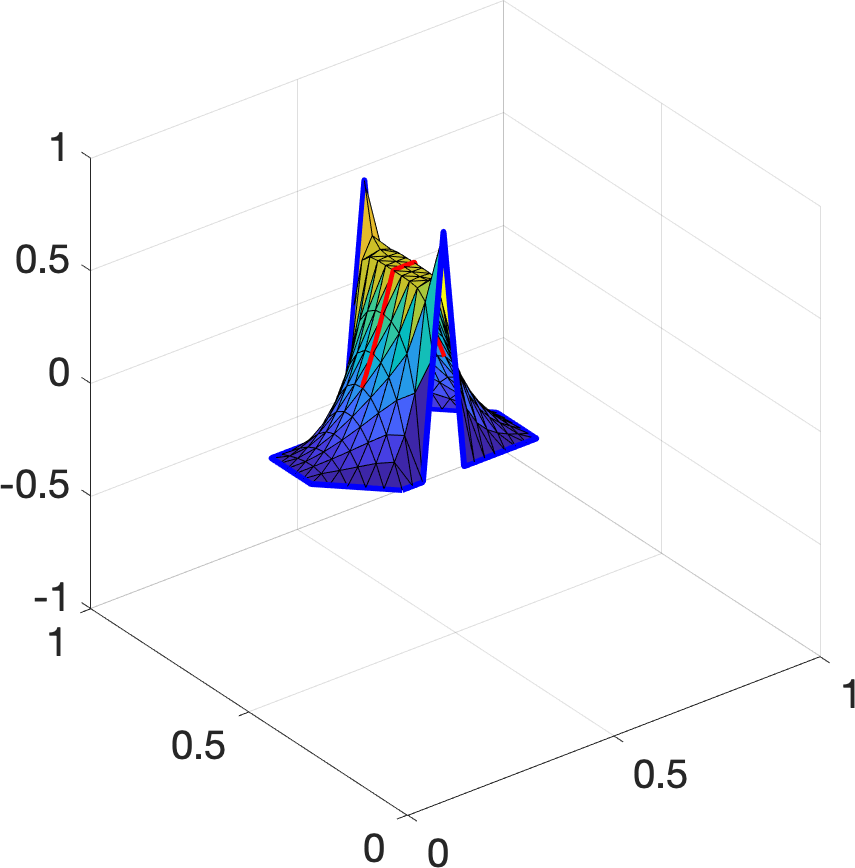}\\
\includegraphics[width=\textwidth,trim=35mm 89mm 13mm 61mm,clip]{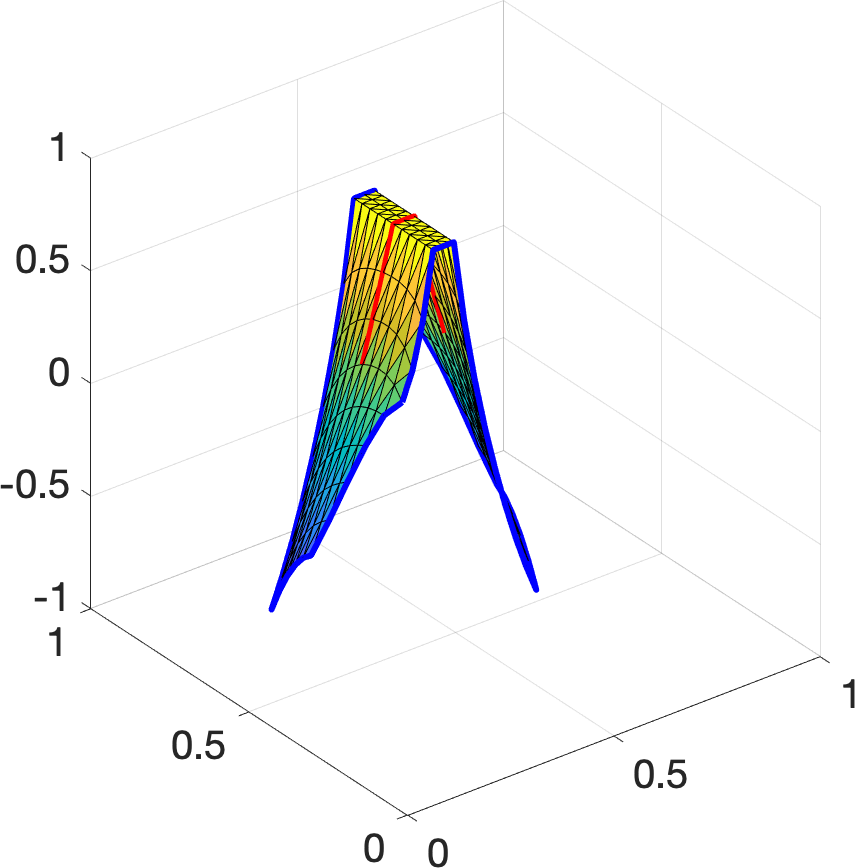}
\end{minipage}
\begin{minipage}{0.24\textwidth} \centering
\includegraphics[width=\textwidth,trim=35mm 89mm 13mm 61mm,clip]{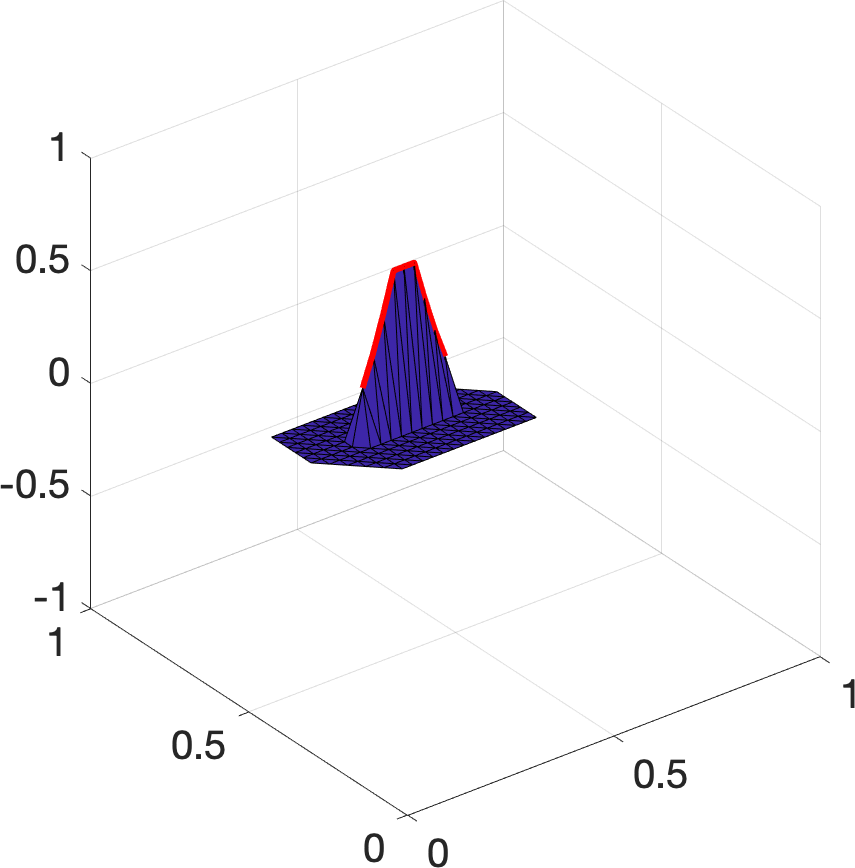} \\
\includegraphics[width=\textwidth,trim=35mm 89mm 13mm 61mm,clip]{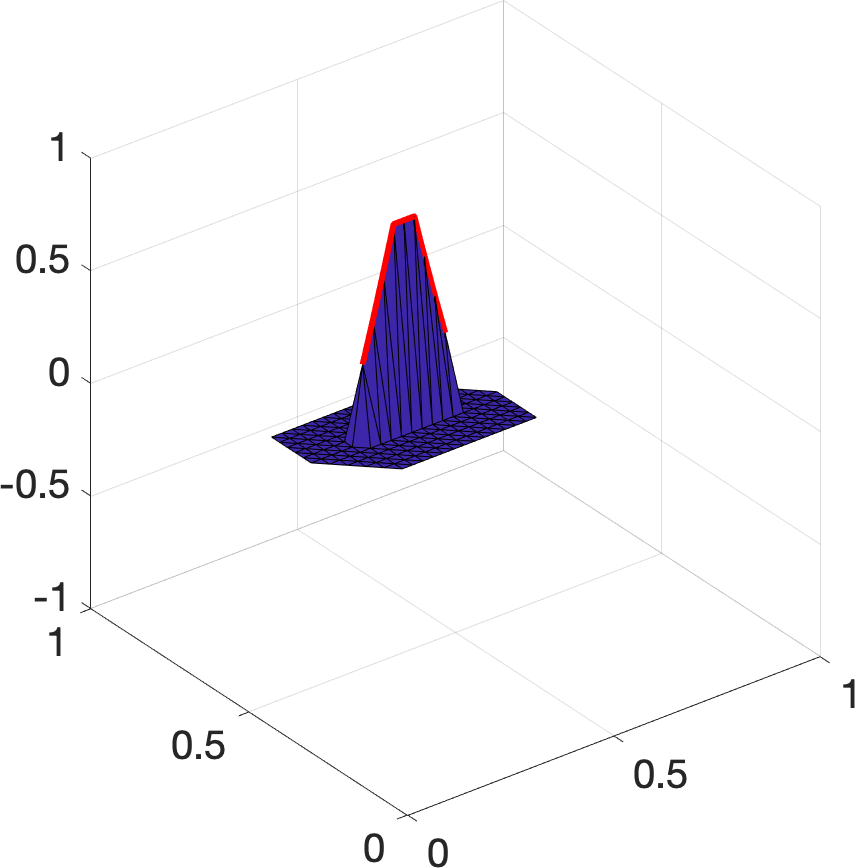}
\end{minipage}
\begin{minipage}{0.24\textwidth} \centering
\includegraphics[width=\textwidth,trim=35mm 89mm 13mm 61mm,clip]{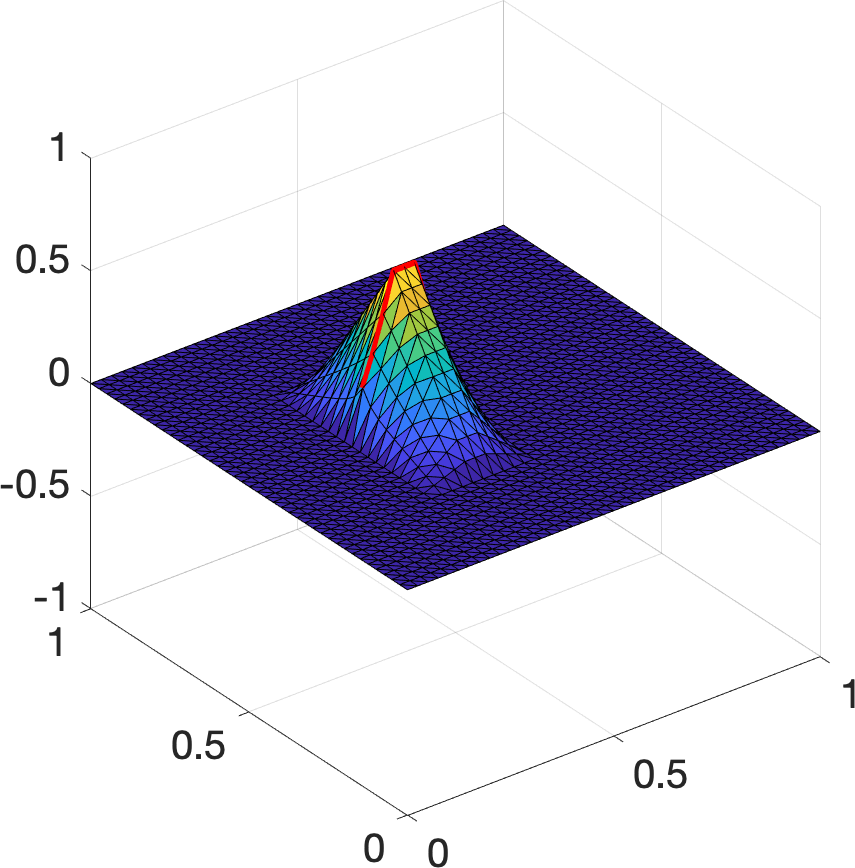} \\
\includegraphics[width=\textwidth,trim=35mm 89mm 13mm 61mm,clip]{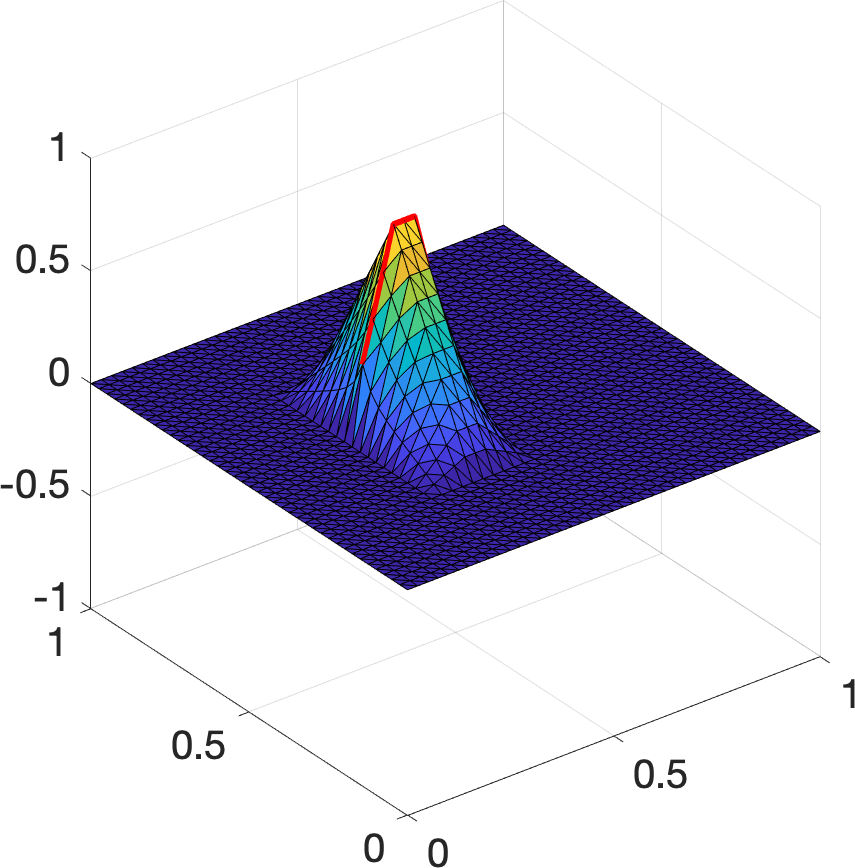}
\end{minipage}

\caption{Simple coefficient distribution with one vertical channel cutting each horizontal edge (left/1st column); the discrete interior edge $\mathring{e}$ and $\partial\Omega_e$ are plotted as solid red and blue lines, respectively. Visualization of the different functions arising in the transfer eigenvalue problem for $X_{\text{VCDT}-l^2}$ (top) and $X_{\text{VCDT}-a}$ (bottom):$H_{e \rightarrow \mOe}^{\partial\Omega_{e}} \varphi_{e}^{(i)}$ (2nd column), $E_{e \to \mOe} \varphi_{e}^{(i)}$ (3rd column), and $H_{\Gamma \rightarrow \Omega} (E_{e \to \Gamma} \varphi_{e}^{(i)})$ (4th column).
\label{fig:comp_evps}}
\end{figure}

It is remarkable that, as observed in~\cref{sec:results:first}, both the $X_{\text{VCDT}-l^2}$ and $X_{\text{VCDT}-a}$ variants yield comparable results when investigating the eigenvalue problems in more detail. In~\cref{fig:comp_evps}, we plot the functions appearing in the eigenvalue problem. Even though the traces of the chosen eigenfunctions on $\partial\Omega_e$ (corresponding to the channel) differ significantly due to the different inner products, the trace on $e$ is almost the same. Hence, the resulting coarse basis functions, which are computed by extending the edge values into the interior, span almost the same space.

\subsection{Coarse space reduction by enlarging $\boldsymbol{\Omega_e}$}

\begin{table}
\begin{center}
\begin{tabular}{|l||l@{\hspace{6pt}}l||r@{\hspace{2pt}}c@{\hspace{2pt}}r|r@{\hspace{6pt}}r|}
	\hline
	$X_0$                                  & $\Omega_{\rm out}$ & $tol_{tr}$   & \multicolumn{3}{c|}{$\dim X_0$} & \multicolumn{1}{c}{$\kappa$} & \multicolumn{1}{c|}{\# its.} \\ \hline\hline
	$X_{\text{GDSW}}$                      & --                 & --           & $33$ & / &                 $33$ &                       $24.1$ &                         $31$ \\ \hline
	\multirow{4}{*}{$X_{\text{AGDSW}}$}    & $\Omega_e^{2h}$    & --           & $57$ & / &                 $57$ &                        $7.1$ &                         $24$ \\
	                                       & $\Omega_e^{5h}$    & --           & $45$ & / &                 $45$ &                       $12.6$ &                         $26$ \\
	                                       & $\Omega_e^{H}$     & --           & $33$ & / &                 $33$ &                       $24.1$ &                         $31$ \\
	                                       & --                 & --           & $33$ & / &                 $33$ &                       $24.1$ &                         $31$ \\ \hline
	\multirow{3}{*}{$X_{\text{VCDT}-l^2}$} & $\Omega_e^{2h}$    & $\bf 10^{6}$ & $57$ & / &                 $69$ &                        $7.1$ &                         $24$ \\
	                                       & $\Omega_e^{5h}$    & $10^{5}$     & $45$ & / &                 $57$ &                       $17.1$ &                         $33$ \\
	                                       & $\Omega_e^{H}$     & $10^{5}$     & $33$ & / &                 $57$ &                       $24.1$ &                         $31$ \\ \hline
\end{tabular}
\end{center}
\caption{Numerical results for the coefficient function shown in~\cref{fig:channels_varying_length_comb} (right), $4 \times 4$ subdomains with $H/h = 10$, and $\alpha_{\max} = 10^6$ and $\alpha_{\min} = 1$ using different coarse spaces; in case of the novel coarse spaces $X_{\text{VCDT}-*}$, we vary the size of $\Omega_e$: two layers of finite elements ($\Omega_e^{2h}$), five layers of finite elements ($\Omega_e^{5h}$), or one layer of subdomains ($\Omega_e^{H}$) around $e$; see also~\cref{fig:channels_varying_length_comb} (right). We report the coarse space dimension (final dimension / dimension before POD orthogonalization), the estimated condition number, and the iteration count. Non-default tolerances $tol_{tr}$ are marked in \textbf{bold face}.
\label{tab:comb}}
\end{table}

In~\cref{sec:results:first}, we already observed the influence of the size of $\Omega_e$ on the spectra of the eigenvalue problems. Here, we discuss a second example, which is visualized in~\cref{fig:channels_varying_length_comb} (right), where the effect is even stronger and better interpretable. 

It can be observed that a single edge function{\AH, and thus a coarse space of dimension $33$,} is sufficient for robustness for this example because there is only a single connected high coefficient component cutting each edge. Consequently, in the results in~\cref{tab:comb}, even the classical GDSW coarse yields good results. This can only be detected by the eigenvalue problem if $\Omega_e$ is large enough to cover this whole high-coefficient component. If $\Omega_e$ is too small, the high coefficient component appears as either three, two, or one component cutting the edge for $\Omega_e^{2h}$, $\Omega_e^{5h}$, or $\Omega_e^{H}$, respectively; cf.~\cref{fig:channels_varying_length_comb} (right). Also {\AH when} using the two subdomains adjacent to the edge $e$ as the domain for the energy minimizing extension, as in the default AGDSW approach, the component is detected as a whole.

This is also reflected clearly in the numerical results in~\cref{tab:comb}. There are $12$ edges cut by such a high-coefficient component. While increasing the size of $\Omega_e$ from $\Omega_e^{2h}$ to $\Omega_e^{5h}$ or $\Omega_e^{H}$ the coarse space dimension reduces by $12$ or $24$, respectively. The same behavior can be observed for AGDSW, when using $\Omega_e^{2h}$, $\Omega_e^{5h}$, and $\Omega_e^{H}$ as the extension domain. This effect has already been reported for similar cases in~\cite{Heinlein:2018:AGD,Heinlein:2019:AGD}.

\subsection{Random coefficient distributions} \label{sec:results:random}

\begin{figure}
\begin{center}
\includegraphics[height=0.32\textwidth]{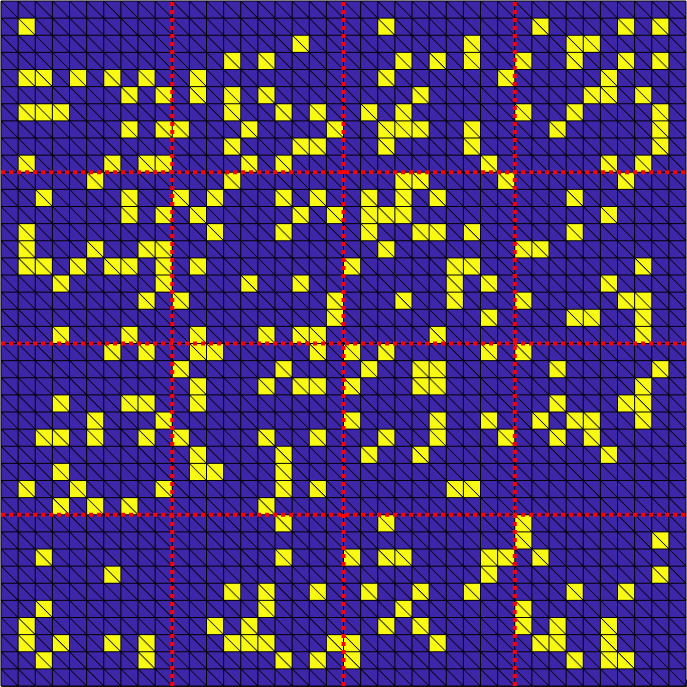}
\includegraphics[height=0.32\textwidth]{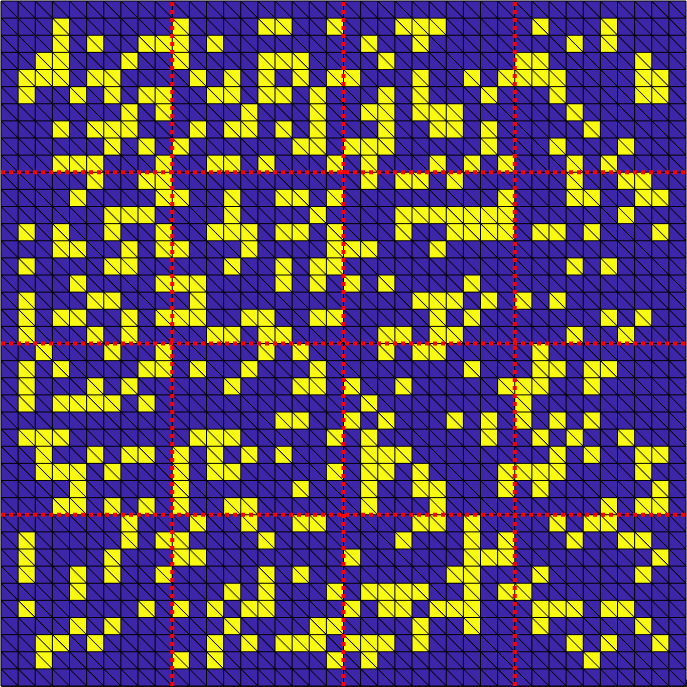}
\includegraphics[height=0.32\textwidth]{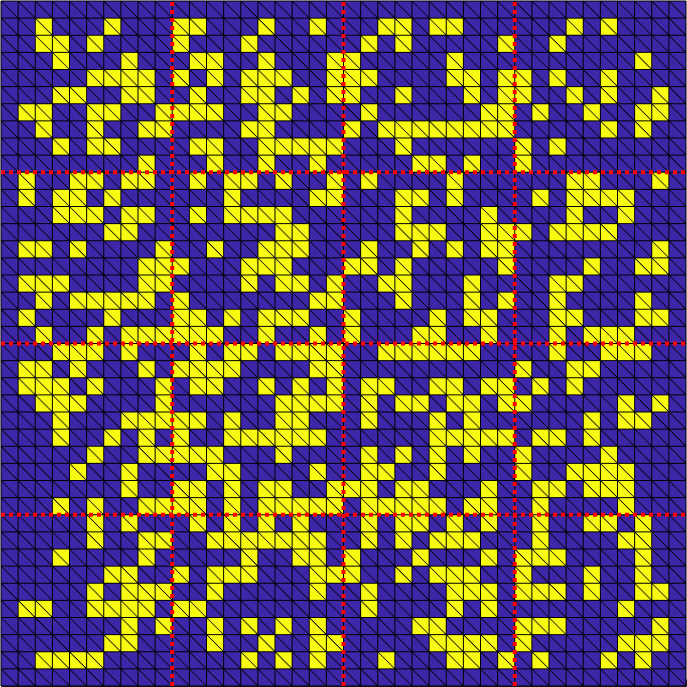}
\end{center}
\caption{Exemplary random heterogeneous coefficient functions with $20\,\%$ (left), $30\,\%$ (middle), and $40\,\%$ (right) elements with high coefficient on a $4 \times 4$ subdomains with $H/h = 10$: the domain decomposition interface is depicted as dashed red lines, the elements with $\alpha = \alpha_{\max}$ are colored yellow, and the elements with $\alpha = \alpha_{\min}$ are colored dark blue.
\label{fig:random}}
\end{figure}

\begin{table}
\begin{center}
\begin{tabular}{|l||l|l||r@{\hspace{1pt}}r@{\hspace{2pt}}c@{\hspace{2pt}}r@{\hspace{1pt}}r|r@{\hspace{1pt}}r@{\hspace{6pt}}r@{\hspace{1pt}}r|}
	\hline
	$\alpha_{\rm max}$        & $\Omega_{\rm out}$               &   $tol_{tr}$ &      \multicolumn{5}{c|}{$\dim X_0$}      &     \multicolumn{2}{c}{$\kappa$}      & \multicolumn{2}{c|}{\# its.} \\ \hline\hline
	\multirow{3}{*}{$20\,\%$} & $\Omega_e^{2h}$                  &     $10^{5}$ &  $85.7$ & $(105)$ & / & $128.5$ & $(150)$ &           $12.1$ &           $(36.0)$ & $30.1$ & $(36)$              \\
	                          & $\Omega_e^{5h}$                  &     $10^{5}$ &  $62.6$ & $(77)$  & / & $127.3$ & $(158)$ &            $9.1$ &           $(27.9)$ & $27.8$ & $(32)$              \\
	                          & $\Omega_e^{H}$                   &     $10^{5}$ &  $62.4$ & $(74)$  & / & $122.3$ & $(142)$ &            $8.6$ &           $(11.6)$ & $27.5$ & $(31)$              \\ \hline
	\multirow{3}{*}{$30\,\%$} & $\Omega_e^{2h}$                  &     $10^{5}$ & $121.6$ & $(143)$ & / & $154.8$ & $(176)$ &           $20.6$ &           $(86.3)$ & $30.0$ & $(41)$              \\
	                          & $\Omega_e^{5h}$                  &     $10^{5}$ &  $70.6$ & $(81)$  & / & $122.7$ & $(143)$ &           $10.6$ &           $(25.5)$ & $27.4$ & $(34)$              \\
	                          & $\Omega_e^{H}$                   &     $10^{5}$ &  $62.9$ & $(74)$  & / & $122.4$ & $(143)$ &           $13.3$ &           $(38.4)$ & $27.6$ & $(37)$              \\ \hline
	\multirow{6}{*}{$40\,\%$} & \multirow{4}{*}{$\Omega_e^{2h}$} & $\bf 10^{7}$ &  $79.9$ & $(87)$  & / &  $81.2$ &  $(88)$ & $1.1 \cdot 10^4$ & $(9.6 \cdot 10^4)$ & $51.4$ & $(105)$             \\
	                          &                                  & $\bf 10^{6}$ & $119.0$ & $(133)$ & / & $125.5$ & $(136)$ &          $223.9$ &       $(1\,879.4)$ & $34.6$ & $(59)$              \\
	                          &                                  &     $10^{5}$ & $155.1$ & $(172)$ & / & $180.7$ & $(200)$ &           $17.2$ &          $(296.0)$ & $25.5$ & $(33)$              \\
	                          &                                  & $\bf 10^{4}$ & $162.3$ & $(179)$ & / & $190.9$ & $(210)$ &            $6.7$ &           $(29.4)$ & $21.7$ & $(26)$              \\ \cline{2-12}
	                          & $\Omega_e^{5h}$                  &     $10^{5}$ &  $81.3$ & $(94)$  & / & $112.3$ & $(126)$ &           $11.5$ &           $(40.6)$ & $27.3$ & $(34)$              \\
	                          & $\Omega_e^{H}$                   &     $10^{5}$ &  $59.0$ & $(68)$  & / &  $95.2$ & $(116)$ &           $23.3$ &           $(76.9)$ & $32.9$ & $(44)$              \\ \hline
\end{tabular}
\end{center}
\caption{Numerical results for randomly distributed coefficient function as shown in~\cref{fig:random} with $20\,\%$, $30\,\%$, and $40\,\%$ high coefficient elements for $4 \times 4$ subdomains with $H/h = 10$, and $\alpha_{\max} = 10^6$ and $\alpha_{\min} = 1$ using different coarse spaces; in case of the novel coarse spaces $X_{\text{VCDT}-*}$, we vary the size of $\Omega_e$: two layers of finite elements ($\Omega_e^{2h}$), five layers of finite elements ($\Omega_e^{5h}$), or one layer of subdomains ($\Omega_e^{H}$) around $e$; see also~\cref{fig:channels_varying_length_comb}; for the $\Omega_e^{2h}$ case, we vary the threshold $tol_{tr}$. We report the coarse space dimension (final dimension / dimension before POD orthogonalization), the estimated condition number, and the iteration count averaged over $100$ runs (maximum number in parentheses). Non-default tolerances $tol_{tr}$ are marked in \textbf{bold face}.
\label{tab:random}}
\end{table}

In order to validate the theory and show robustness {\AH of our fully algebraic approach ($X_{\text{VCDT}-l^2}$ coarse space)} for general coefficient distributions, we test it on randomly distributed binary coefficient distributions; examples for {\AH coefficient distributions with} $20\,\%$, $30\,\%$, and $40\,\%$ elements with high coefficients are shown in~\cref{fig:random}. 

The results are listed in~\cref{tab:random}, and we can draw several conclusions from those results{\AH: First,} we generally obtain good convergence for our fully algebraic approach for different ratios of high coefficient elements and sizes of $\Omega_e$. As we already observed before, enlarging $\Omega_e$ reduces the coarse space dimension, which is much more pronounced for larger ratios of high coefficient elements; for instance, for $40\,\%$ of high coefficient elements, the coarse space dimension can be reduced from $155.1$ to $59.0$ on average, when keeping the tolerances fixed. 

Of course, enlarging $\Omega_e$ also increases the computational work for setting up both eigenvalue problems. As an alternative, we consider the smallest $\Omega_e = \Omega_e^{2h}$ and vary the threshold for the transfer eigenvalue problem for the tolerance: when increasing $tol_{tr}$ from $10^4$ to $10^6$, the coarse space dimension reduced from $162.3$ to $119.0$. At the same time, the condition number and iteration count increase moderately: the maximum iteration count goes up from $34$ to $44$ and the maximum condition number from $29.4$ to $1\,879.4$. When increasing the tolerance further to $10^7$, we obtain an even smaller coarse space dimension of $79.9$; however, the maximum condition number and iteration count deteriorate to $9.6 \cdot 10^4$ and $105$, respectively.
Obtaining robustness using the fully algebraic coarse space depends on an interplay of the hyper parameters of the method, such as the size of $\Omega_e$ and the tolerances; a full investigation is not possible {\AH here} due to space limitations.

\subsection{SPE10 model problem} \label{sec:results:spe10}

\begin{table}
\begin{center}
\begin{tabular}{|l||l@{\hspace{6pt}}l||r@{\hspace{2pt}}c@{\hspace{2pt}}r|r@{\hspace{6pt}}r|}
	\hline
	$X_0$                                  & $\Omega_{\rm out}$               & $tol_{tr}$   & \multicolumn{3}{c|}{$\dim X_0$} & \multicolumn{1}{c}{$\kappa$} & \multicolumn{1}{c|}{\# its.} \\ \hline\hline
	\multicolumn{8}{|c|}{Original coefficient (without thresholding)}                                                                                                                                                        \\ \hline
	$X_{\text{GDSW}}$                      & --                               & --           &  $85$ & / &                $85$ &                       $20.6$ &                         $42$ \\ \hline\hline
	\multicolumn{8}{|c|}{Binary coefficient (with thresholding)}                                                                                                                                                             \\ \hline
	$X_{\text{GDSW}}$                      & --                               & --           &  $85$ & / &                $85$ &             $2.0 \cdot 10^5$ &                         $57$ \\ \hline
	$X_{\text{AGDSW}}$                     & --                               & --           &  $93$ & / &                $93$ &                       $19.3$ &                         $38$ \\ \hline
	\multirow{5}{*}{$X_{\text{VCDT}-l^2}$} & \multirow{3}{*}{$\Omega_e^{2h}$} & $\bf 10^{7}$ & $147$ & / &               $150$ &                     $1859.0$ &                         $40$ \\
	                                       &                                  & $\bf 10^{6}$ & $262$ & / &               $273$ &                      $122.8$ &                         $37$ \\
	                                       &                                  & $10^{5}$     & $362$ & / &               $417$ &                        $9.3$ &                         $31$ \\ \cline{2-8}
	                                       & $\Omega_e^{5h}$                  & $10^{5}$     & $191$ & / &               $229$ &                        $9.3$ &                         $31$ \\
	                                       & $\Omega_e^{H}$                   & $10^{5}$     & $147$ & / &               $176$ &                        $9.6$ &                         $31$ \\ \hline
	\multirow{3}{*}{$X_{\text{VCD}}$}      & $\Omega_e^{2h}$                  & --           &  $87$ & / &                $89$ &             $2.0 \cdot 10^5$ &                         $57$ \\
	                                       & $\Omega_e^{5h}$                  & --           &  $90$ & / &                $92$ &                       $19.4$ &                         $39$ \\
	                                       & $\Omega_e^{H}$                   & --           &  $90$ & / &                $93$ &                       $19.4$ &                         $39$ \\ \hline
\end{tabular}
\end{center}
\caption{Numerical results for the coefficient functions shown in~\cref{fig:spe10} (with and without thresholding), $6 \times 6$ subdomains with $H/h = 10$ using different coarse spaces; in case of the novel coarse spaces $X_{\text{VCDT}-*}$, we vary the size of $\Omega_e$: two layers of finite elements ($\Omega_e^{2h}$), five layers of finite elements ($\Omega_e^{5h}$), or one layer of subdomains ($\Omega_e^{H}$) around $e$; see also~\cref{fig:channels_varying_length_comb}; for the $\Omega_e^{2h}$ case, we vary the threshold $tol_{tr}$. We report the coarse space dimension (final dimension / dimension before POD orthogonalization), the estimated condition number, and the iteration count. Non-default tolerances are marked in \textbf{bold face}.
\label{tab:spe10}}
\end{table}

\begin{figure}
	\begin{center}
		\includegraphics[width=0.24\textwidth]{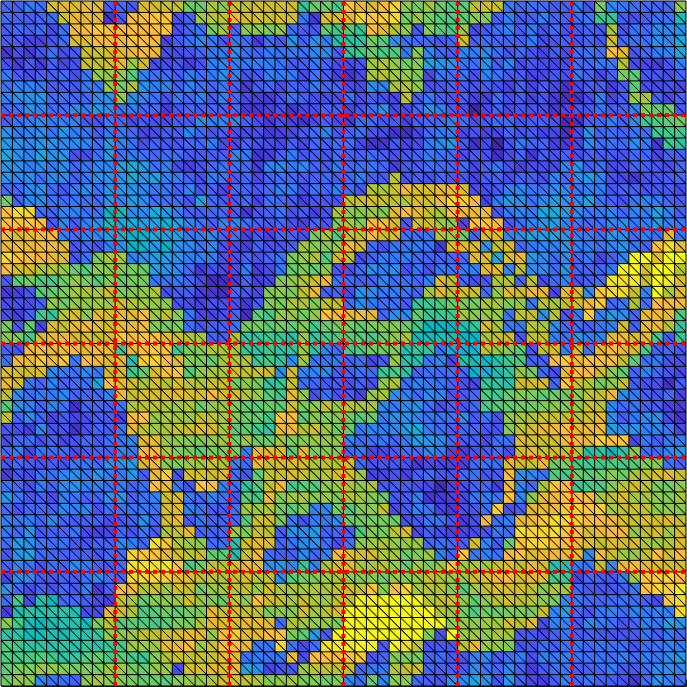}\quad
		\includegraphics[width=0.24\textwidth]{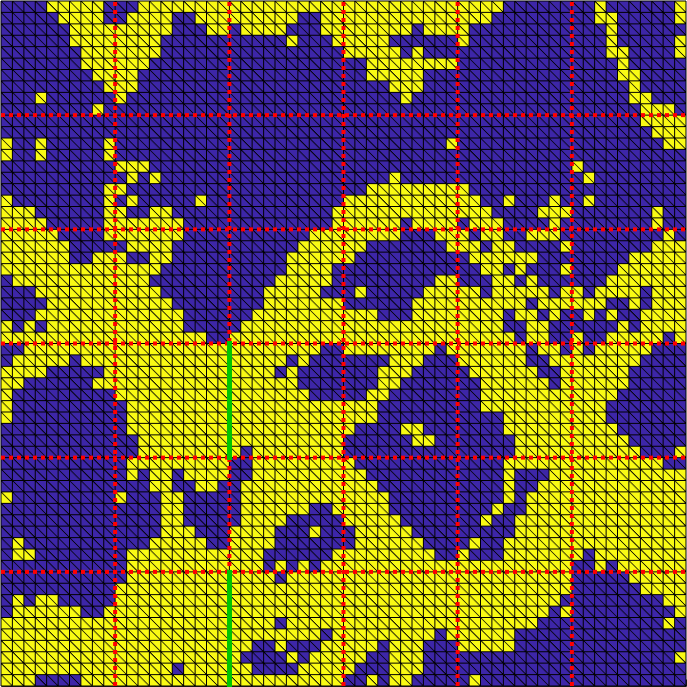}
	\end{center}
	\caption{\lila{Heterogeneous coefficient functions on $6 \times 6$ subdomains with $H/h = 10$; the domain decomposition interface is depicted as dashed red lines. \textbf{Left:} Coefficient values are between $\alpha_{\max} = 2.9\cdot^4$ (yellow) and $\alpha_{\min} = 5.8\cdot^{-3}$ (dark blue). \textbf{Right:} Binary heterogeneous coefficient function created with a threshold of $1$: every coefficient above $1$ is mapped to $\alpha_{\max} = 10^6$, and every coefficient below $1$ is mapped to $\alpha_{\min} = 1$.
			\label{fig:spe10}}}
	
	\vspace{-15pt}
\end{figure}

Finally, we consider a coefficient function based on realistic data. In particular, we use heterogeneous coefficient functions $\alpha$ generated from \lila{parts of} the $40$th layer of the second data set from the 2001 SPE Comparative Solution Project benchmark~\cite{christie2001tenth}, \lila{employing the pixel-wise norm of the permeabilities as the coefficient function.} 
As can be observed in~\cref{tab:spe10} (\textit{``Original coefficient (without thresholding)''}), this example can be solved robustly using the classical GDSW coarse space, and no adaptive coarse space is needed. However, if we convert $\alpha$ into a binary coefficient function by setting all coefficients above $1.0$ to $\alpha_{\max} = 10^6$ and all coefficients below $1.0$ to $\alpha_{\min} = 1$, the classical GDSW coarse space is not robust anymore, resulting in a high condition number of $2.0 \cdot 10^5$.

As expected, the $X_{\text{AGDSW}}$ and $X_{\text{VCDT}-l^2}$ adaptive coarse spaces yield robust results. For small sizes of $\Omega_e$, the dimension of the $X_{\text{VCDT}-l^2}$ coarse space is quite high; for instance, the coarse space dimension is $362$ for $\Omega_e = \Omega_e^{2h}$ (and $tol_{tr} = 10^5$). We observe that, the coarse space dimension reduces significantly when increasing $tol_{tr}$; for $tol_{tr} = 10^7$, the dimension is only $147$. The same dimension is obtained for $\Omega_e = \Omega_e^H$. While enlarging $\Omega_e$ results in a better condition number and iteration count, it also increases the computational cost for setting up the eigenvalue problems. On the other hand, increasing $tol_{tr}$ does not increase the computational cost; however, the condition number and iteration count grow moderately.

For this example, we also report results for the space $X_{\text{VCD}}$, where the transfer eigenvalue problem is completely neglected. While the condition number is contrast dependent for $\Omega_e = \Omega_e^{2h}$, which shows that the transfer eigenvalue problem is necessary in this case, we obtain good results for $\Omega_e = \Omega_e^{5h}$ and $\Omega_e = \Omega_e^{H}$, and the dimension of the coarse space is even lower than the dimension of $X_{\text{AGDSW}}$. Note that for oversampling domains where the coefficient function is high nearly everywhere, we conjecture a slower decay of the harmonic extensions of higher frequency modes on $\pOe$. This results in a relatively large {\AH $X_{\text{VCT}-l^2}$} space on $e$; this is the case, for  example, for the green {\AH vertical} edges in~\cref{fig:spe10}.

These results show that our fully algebraic approach is very robust, but further investigations of choosing the thresholds will be necessary to obtain the optimal coarse space dimension.

\FloatBarrier

\bibliographystyle{siamplain}

\end{document}